\documentclass[a4paper,draft]{amsart}

\usepackage{a4wide}
\usepackage{amsmath,amssymb,amsthm,calrsfs}
\usepackage[UKenglish]{babel}
\usepackage[usenames]{color}

\newcounter{kommentar}

\DeclareMathAlphabet{\mathscr}{OT1}{pzc}{m}{it}

\numberwithin{equation}{section}

\def\a{\alpha} \def\b{\beta} \def\d{\delta} \def\e{\epsilon} \def\f{\varphi}
\def\l{\lambda}  \def\R{\mathbb{R}} 
 
 \def\I{\mathbb{I}}
 
\def\({\left(} \def\){\right)} 
\def\<{\langle} \def\>{\rangle}

\def\inv{^{-1}}
\def\N{\mathbb{N}}

\renewcommand\ge{\geqslant}
\renewcommand\le{\leqslant}
\newcommand\ie{i.e.}

\newcommand\eg{e.g.}

\newcommand\forget[1]{}

\DeclareMathOperator{\im}{im}

\newcommand{\Q}{\mathbb{Q}}

\newenvironment{smatrix}{\left(\begin{smallmatrix}}{\end{smallmatrix}\right)}

\DeclareMathOperator{\tild}{\sim\!}
\newcommand{\impl}{\Rightarrow}
\newcommand{\equ}{\Leftrightarrow}

\newtheorem{thm}{Theorem}[section]
\newtheorem{lma}[thm]{Lemma}
\newtheorem{prop}[thm]{Proposition}
\newtheorem{cor}[thm]{Corollary}

\theoremstyle{definition}
\newtheorem{dfn}[thm]{Definition}
\theoremstyle{remark}
\newtheorem{rmk}[thm]{Remark}

\theoremstyle{definition}
\newtheorem{ex}[thm]{Example}

\newcommand\Sg{\mathop{\mathbf{Sg}}}
\newcommand\Sets{\mathop{\mathbf{Sets}}}
\newcommand\aSg{\mathop{\mathbf{Sg}^\#}}
\newcommand\aSets{\mathop{\mathbf{Sets}^\#}}
\newcommand\apt{\mathbin{\#}}
\DeclareMathOperator{\cker}{coker}
\newcommand\eller{\,\vee\,}
\newcommand\och{\,\wedge\,}

\newcommand\Lrel{\mathbin{\mathcal{L}}}
\newcommand\Rrel{\mathbin{\mathcal{R}}}
\newcommand\Hrel{\mathbin{\mathcal{H}}}
\newcommand\Drel{\mathbin{\mathcal{D}}}
\newcommand\Jrel{\mathbin{\mathcal{J}}}
\newcommand\lrel{\mathbin{\mathfrak{L}}}
\newcommand\rrel{\mathbin{\mathfrak{R}}}
\newcommand\hrel{\mathbin{\mathfrak{H}}}
\newcommand\drel{\mathbin{\mathfrak{D}}}
\newcommand\jrel{\mathbin{\mathfrak{J}}}

 \renewcommand{\labelenumi}{\arabic{enumi})}

\def\Z{\mathbb{Z}}

\title{Some results in constructive semigroup theory}
\author{Erik Darp\"o}
\address[Darp\"o]{Graduate School of Mathematics, Nagoya University, Furo-cho, Chikusa-ku, Nagoya, Japan}

\author{Melanija Mitrovi\'c}
\address[Mitrovi\'c]{Faculty of Mechanical Engineering, University of Ni\v{s}, Serbia}

\begin{document}
\selectlanguage{UKenglish}

\begin{abstract}
  We give a constructive treatment of some basic concepts and results in semigroup theory. Focusing on semigroups equipped with an apartness relation, we give analogues, from the point of view of apartness, of several classical constructions and results, such as transitive closure and congruence closure, free semigroups, periodicity, Rees factors, and Green's relations.
\end{abstract}

\maketitle

\tableofcontents

%%%%%%%%%%%%%%%%%%%%%%%%%%%%%%%%%%%%%%%%%%%%%%%%%%%%%%%%%%%%%%%%%%%%%%%%%%%%%%
\section{Introduction}

%%%%%%%%%%%%%%%%%%%%%%%%%%%%%%%%%%%%%%%%
\subsection{The constructive framework}

The purpose of this article is to give a treatment of some fundamental topics in semigroup theory, in the framework of constructive mathematics. This means that we use intuitionistic instead of classical logic; in particular, the \emph{law of excluded middle} (LEM): $P \vee \neg P$ is disallowed as a general principle. In line with the approach taken by Erret Bishop \cite{bishop67}, we strive to follow the classical treatment as closely as possible, and all our results are compatible with the classical theory.

The constructive framework features certain distinctions that are not present in the classical case. One, which is central to this work, is the one between \emph{inequality} and \emph{apartness}. While inequality is simply the negation of equality, apartness can be understood as a positive, constructive statement of two entities being different. For example, two binary sequences $a=(a_m)_{m\in\N}$ and $b=(b_n)_{n\in\N}$ are said to be apart if there exists an $n\in\N$ such that $a_n\ne b_n$. This is a stronger statement than $\neg\forall_{n\in\N}(a_n=b_n)$, since its proof requires the explicit construction of a natural number $n$ such that $a_n\ne b_n$, whilst the latter is merely the statement that $\forall_{n\in\N}(a_n=b_n)$ is impossible.

In this paper, the main object of study is \emph{semigroups with apartness}, which are semigroups equipped with an apartness relation satisfying a compatibility condition with the binary operation.
One of the advantages with this setup is that the axiom of \emph{cotransitivity} for apartness relations allows us make arguments by disjunction in many cases where this would otherwise not have been possible.
Some classically valid statements have constructive analogues that are most naturally formulated in terms of the apartenss. A simple example of this is the statement that $(\R,\cdot)$ is a zero-group (i.e., of the form $G\cup\{0\}$ for some group $G$). While this result is not provable in the constructive setting, the classically equivalent statement that the set of real numbers that are \emph{apart} from zero form a group under multiplication, still holds.  
Many semigroups and other algebraic structures, including $\N$, $\Q$, $\R$ and the set $\{0,1\}^{\N}$ of binary sequences, come equipped with a natural notion of apartness.

Following Bishop, we think of a set as the totality of all objects obtained through some specified construction, together with a notion of \emph{equality}, which may be any equivalence relation (this corresponds to the type-theoretic notion of a \emph{setoid}).
All predicates, functions etc.{} considered are assumed to be \emph{extensional}, that is, invariant under equality.
In this setup, a factor set $X/\e$ of a set $X$ by some (extensional) equivalence relation $\e$ is simply defined as the set $X$ with $\e$ taken as the equality relation.

Throughout the text, we try to outline the limits of our approach by giving \emph{weak counterexamples} to (classically valid) statements that cannot be proved constructively.
This means to prove that the result in question
implies some known non-constructive statement, such as LEM, or the \emph{limited principle of omniscience} (LPO), which is the statement that for each binary sequence $a=(a_n)_{n\in\N}$, either $\forall_{n\in\N}(a_n=0)$ or $\exists_{n\in\N}(a_n=1)$ holds.
Another non-constructive statement that we will employ for this purpose is the \emph{weak law of excluded middle} (WLEM): for every proposition $P$, either $\neg P$ or $\neg\neg P$ holds.

The constructive study of general algebraic structures was pioneered by Heyting \cite{heyting27,heyting41}, although its origins can be traced further back to Kronecker \cite{kronecker1882} and, in some sense, even Gauss.
Heyting gave constructive treatments of basic structures including groups, rings and field, equipped with an apartness relation. The notion of a co-substructure, which plays an important role in the present work, was introduced by Scott \cite{scott79}, drawing on ideas from Heying.
Treatments of the basic notions and theory of constructive algebra in are given in the books \cite{mrr88} and \cite{tvd88II}, summarising work by many authors. There is also a very extensive treatment of commutative algebra by Lombardi and Quitt\'{e} \cite{lq15}, and recent work on central simple algebras \cite{qln21}, featuring a constructive version of Wedderburn's structure theorem.
Much owing to the work of Crnvenkovi\'c, Mitrovi\'c and Romano, a constructive theory of semigroups with apartness has begun to emerge in the last decades, with contributions including \cite{cf19,cmr13,romano99,romano02}.

This paper is organised as follows.
Section~\ref{fundamental} contains some basic concepts and notation, and Section~\ref{preliminaries} some results about sets and relations that will be needed later on. In Section~\ref{kernels}, we give constructions of the cotransitive kernel and the co-congruence kernel of a relation. These are related to the classically well-known constructions of transitive closure and congruence closure in a way similar to how apartness is related to equality.
Section~\ref{semigroups} contains constructive treatments of free, monogenic and periodic semigroups, idempotents, and the Rees congruence.
Finally, in Section~\ref{greensrel}, we study Green's relations, first making a constructive walk-through of the classical theory (Section~\ref{greenclassical}), then defining and studying analogues (so-called ``constructive friends'') of these relations based on the apartness -- instead of equality -- relation (Section~\ref{greenfriends}).

In Section~\ref{greenfriends}, we make a significant compromise about the constructive framework: we assume that the \emph{constant domains principle} \eqref{cdaxiom} (see Section~\ref{filledcd})
holds for predicates $P(x)$ and $Q$ of certain types, defined in terms of the apartness relation.
The restrictiveness of this assumption varies depending on the set and the apartness relation. Applied to the multiplicative semigroup $\R$ with the natural apartness it entails LPO, whereas in other cases, for example for $\N$, it does not at all infringe upon the constructive validity of the results.

%%%%%%%%%%%%%%%%%%%%%%%%%%%%%%%%%%%%%%%%
\subsection{Fundamental concepts} \label{fundamental}

Here, we summarise the main definitions and concepts used in this paper. 
References include \cite{mrr88,tvd88I,tvd88II} for general concepts, and \cite{cmr13,romano99,romano02} for semigroup-specific ones. 

A (binary) relation between sets $X$ and $Y$ is a subset $\alpha\subset X\times Y$. In particular, trivially, $(=) = \{(x,x)\mid x\in X\}\subset X\times X$.
The \emph{negation}, or \emph{logical complement} of a subset $Y\subset X$ is the set $\neg Y = \{x\in X \mid x\notin Y\} = \{x\in X \mid x\in Y \impl \bot \}$.
For a relation $\alpha$, we often write $x\a y$ instead of $(x,y)\in\a$. Thus, in particular, $x(\neg\a)y$, $\neg(x\a y)$ and $(x,y)\notin\a$ all mean the same thing, namely that $(x,y)\in\a$ is impossible.

As usual, for relations $\a\subset X\times Y$ and $\b\subset Y\times Z$, we write
\begin{align*}
  \a\inv &=\{(y,x)\in Y\times X\mid (x,y)\in\a \},\;\mbox{and}\\
  \a\circ\b &= 
  \{(x,z)\in X\times Z\mid \exists_{y\in Y}(x\a y \,\wedge\, y\b z)\}\subset X\times Z \,.
\end{align*}

We view $0$ as the smallest natural number.
An element $x\in X$ is denoted by $[x]$ when considered as a member of the factor set $X/\e$ with respect to some equivalence relation $\e\subset X\times X$.
  
\begin{dfn}
\begin{enumerate}
\item An \emph{apartness relation} on a set $X$ is a relation $(\apt)\subset X\times X$ satisfying
  \begin{enumerate}
  \item $\forall_{x\in X}\neg(x\apt x)$\quad (irreflexivity);
  \item $\forall_{x,y\in X}(x\apt y \impl y\apt x)$\quad (symmetry);
  \item $\forall_{x,y,z\in X}[\: x\apt z \impl(x\apt y \eller y\apt z) \:]$ \quad
    (cotransitivity).
  \end{enumerate}
\item An apartness relation $\apt$ is \emph{tight} if $\neg(\apt )=(=)$, and
\item \emph{standard} if $\neg((\apt) \cup (=)) = \emptyset$.
\item A set $X$ with apartness $\apt$ is \emph{discrete} if $\forall_{x,y\in X}(x=y\eller x\apt y)$.
\item Let $X$ be a set with apartness, and $Y\subset X$. The \emph{apartness complement} of $Y$ in $X$ is the subset $\tild Y = \{x\in X\mid \forall_{y\in Y}(x\apt y)\}\subset X$.
\end{enumerate}
\end{dfn}

\begin{ex}
  \begin{enumerate}
  \item The sets $\N$, $\Z$, $\Q$ and $\{1,2,\ldots,n\}$, with denial apartness: $x\apt y \equ \neg(x=y)$, are discrete.
  \item The set $\{0,1\}^{\N}$ of binary sequences has a tight apartness given by $a\apt b \equ \exists_{n\in \N}(a_n\ne b_n)$. Discreteness of this set with apartness is precisely LPO, and hence cannot be proved constructively. The statement $\neg(=) = (\apt)$ or, equivalently,
    \[ \neg\forall_n(a_n=0) \quad\impl\quad \exists_n(a_n=1), \]
    known as \emph{Markov's principle} (MP), is also not considered to be constructively valid.
  \item The set $\R$ of real numbers has a tight apartness relation defined by $x\apt y \equ (x<y \eller y<x)$. Discreteness of $(\R,\apt)$ is equivalent with LPO \cite[\textsection 4.5]{bridges94}.
  \end{enumerate}
\end{ex}

Throughout this paper, unless otherwise stated, $X$, $Y$, etc.{} denote sets equipped with an apartness relation $\apt$. 
Every subset $A\subset X$ inherits an apartness relation from $X$.
The Cartesian product $X\times Y$ has an apartness relation defined by $(x_1,y_1)\apt(x_2,y_2) \:\equ\: (x_1\apt x_2 \eller y_1\apt y_2)$.

\begin{rmk} \label{apartrmk}
  \begin{enumerate}
  \item An apartness relation $\apt$ is standard if and only if $\mathrm{LEM}\vdash (\apt) = \neg(=)$. In practice, non-standard apartness relations are of limited interest to us. 
    \label{apartrmk1}
  \item Some authors, notably Troelstra and van-Dalen \cite{tvd88II}, include tightness in the definition of an apartness relation. We do not make this assumption here, partly for the reason that some constructions, such as the Rees factor semigroup (see Section~\ref{rees}) naturally give rise to apartness relations that cannot be proved to be tight in general, even when starting from a semigroup with tight apartness.
    \label{apartrmk2}
  \item The following formulae hold for all subsets $A,B\subset X$:
     \[\tild A\subset\neg A, \quad
     \tild\tild\tild A = \tild A, \quad
     \tild\,(A\cup B) = (\tild A)\cap (\tild B), \quad
     (\tild A)\cup (\tild B) \subset \tild\,(A\cap B) \,.\]
     Note also that $A\subset \tild B \:\equ\: B\subset \tild A $, and that $\tild\,(=) = (\apt)$.
     \label{apartrmk3}
  \end{enumerate}
\end{rmk}

Given the central role played by the apartness relation, subsets and functions that are well-behaved with respect to this relation are of particular importance for us.
\begin{dfn}
  \begin{enumerate}
  \item A function $f:X\to Y$ is \emph{strongly extensional} if $f(x)\apt f(y)$ implies $x\apt y$ for all $x,y\in X$; and
  \item \emph{apartness injective} if $x\apt y$ implies $f(x)\apt f(y)$, for all $x,y\in X$.
  \item A subset $A\subset X$ is \emph{strongly extensional} if the implication $a\in A \:\impl\: (x\in A \eller x\apt a)$ holds for all $x,a\in X$.
  \end{enumerate}
\end{dfn}

The term ``strongly extensional'' may be justified by the following observation: for tight apartness relations, the contrapositives of the implications in (1) and (3) are extensionality of the function $f$ and of the subset $\neg A\subset X$, respectively.
Cf.{} \cite{ruitenburg91} and \cite[Sec{.}~8.1]{tvd88II}.
The preimage of a strongly extensional subset under a strongly extensional function is again strongly extensional; see Section~\ref{atopology}.

We denote by $\Sets$ the category of sets and functions, and by $\aSets$ the category of sets with apartness, and strongly extensional functions.

\begin{rmk}
\begin{enumerate}
\item
  A (strongly extensional) function $f:X\to Y$ is injective if and only if it is a  monomorphism in $\Sets$ (in $\aSets$), and surjective if and only if it is an epimorphism in $\Sets$ (in $\aSets$) \cite[Section~4]{mrr88}.
\item
  Any isomorphism in $\Sets$ is a bijective function and, by the axiom of unique choice, the converse holds as well. Hence, by (1), a morphism in $\Sets$ is an isomorphism if and only if it is mono and epi. The corresponding statement does not hold in $\aSets$: there are strongly extensional bijections whose inverses (in $\Sets$) cannot be proved to be strongly extensional.
  A morphism in $\aSets$ is invertible if and only if it is bijective and apartness injective.
\item The apartness relation $(\apt)\subset X\times X$ is strongly extensional (see Corollary~\ref{kappanegcompl}).
  The equality relation $(=)\subset X\times X$ is strongly extensional if and only if $X$ is discrete.
\end{enumerate}
\end{rmk}

\begin{dfn}
 A relation $\a\subset X\times X$ is
  \begin{enumerate}
  \item \emph{strongly irreflexive} if $\a\subset(\apt )$;
  \item \emph{cotransitive} if $\forall_{x,y,z\in X}[\: x\a z \impl(x\a y \eller y\a z) \:]$;
  \item a \emph{co-quasiorder} if it is strongly irreflexive and cotransitive;
  \item a \emph{coequivalence} if it is strongly irreflexive, symmetric and cotransitive.
  \end{enumerate}
\end{dfn}

Co-quasiorders and coequivalences can be thought of as constructive friends, that is, apartness analogues, of quasiorders respectively equivalences. 
The negation $\neg\kappa$ of a co-quasiorder $\kappa$ is a quasiorder (i.e., a reflexive and transitive relation), and the negation of a coequivalence is an equivalence. However, it is in not possible to prove in general that the negation of a quasiorder/equivalence is a co-quasiorder/coequivalence.

\begin{dfn}
\begin{enumerate}
\item A set $C\in\Sets$ is
  \begin{enumerate}
  \item \emph{inhabited} if it contains some element $c\in C$;
  \item \emph{finite} if there exists a natural number $n\in\N$ and an invertible function $\{1,\ldots,n\}\to C$;
  \item \emph{subfinite} if there exists a finite set $B$ and an injective map $C\to B$;
  \item \emph{finitely enumerable} if there exists a finite set $B$ and a surjective map $B\to C$.
  \end{enumerate}
\item A subset $A\subset X$ is 
  \begin{enumerate}
\item \emph{detachable} if $\forall_{x\in X}(x\in A \eller x\notin A)$;
\item \emph{stable} if $\neg\neg A = A$.
  \end{enumerate}
\end{enumerate}
\end{dfn}

\begin{rmk}
\begin{enumerate}
\item Detachable subsets behave like ``classical'' objects, and are therefore of limited interest to us. Stable subsets and relations, on the other hand, occur frequently and naturally in constructive mathematics; for example, the equality relation on any set with a tight apartness is stable:
$\neg\neg(=) = \neg\neg\neg(\apt) = \neg(\apt) = (=)$.
\item Every finitely enumerable strongly extensional subset of $X$ is detachable.
  \end{enumerate}
\end{rmk}

%%%%%%%%%%%%%%%%%%%%%%%%%%%%%%%%%%%%%%%%

\begin{dfn}
\begin{enumerate}
\item A \emph{semigroup with apartness} is a semigroup $S$ equipped with an apartness relation, such that the multiplication $S\times S\to S,\:(x,y)\mapsto xy$ is strongly extensional.
\item A relation $\a$ on $S$ \emph{compatible with multiplication} if
  $\forall_{s,t,u,v\in S}(\, s\a t \wedge u\a v \,\impl\, (su)\a(tv)\,)$.
 It is \emph{left compatible} if $\forall_{s,t,u\in S}(\, s\a t \,\impl\, (us)\a(ut)\,)$.
  If $\a$ in addition is an equivalence relation, it is said to be a \emph{congruence}, respectively a \emph{left congruence}.
\item A relation $\zeta$ on $S$ is \emph{left co-compatible} with the multiplication if $(ax)\zeta(ay) \:\impl\: x\zeta y$, and \emph{co-compatible} with multiplication if 
$(ax)\zeta(by) \:\impl\: a\zeta b \eller x\zeta y$, for all $a,b,x,y\in S$.
\item A coequivalence which is co-compatible with the multiplication on $S$ is a
  \emph{co-congruence}.
\item A subset $A$ of a semigroup $S$ is \emph{right convex} if 
  $ab\in A \:\impl\: b\in A$, and \emph{convex} if it is both left and right convex.
  It is a \emph{(left) co-ideal} if (right) convex and strongly extensional, and a \emph{co-subsemigroup} if it is strongly extensional and $ab\in A  \:\impl\: a\in A \eller b\in A$. 
  \item An ideal in a semigroup $S$ is \emph{finitely generated} if it is generated by a finitely enumerable subset.
\end{enumerate}
\end{dfn}

We denote by $\Sg$ the category of semigroups, and by $\aSg$ the category of semigroups with apartness and strongly extensional morphisms.

Let $f:S\to T$ be a morphism in $\aSg$. In analogy with the classical case, the preimage under $f$ of a (left/right/two-sided) co-ideal in $T$ is a (left/right/two-sided) co-ideal in $A$, and the preimage of a co-subsemigroup of $B$ is a co-subsemigroup of $A$. However, the image under $f$ of a co-subsemigroup of $A$ need not be a co-subsemigroup of $B$.

For $S\in\aSg$, let $S^1=S\mathrel{\dot{\cup}}\{1\}$ (disjoint union) with $1\cdot s = s\cdot 1=s$ and $1\apt s$ for all $s\in S$.
Let $\rho_a,\lambda_a:S\to S,\:\rho_a(x)=ax, \:\lambda_a(x) =ax$ be the maps of right, respectively left, multiplication with an element $a\in S$.
It is straightforward to verify that the strong extensionality of the multiplication on $S$ is equivalent with the maps $\rho_a:S\to S$ and $\l_a:S\to S$ being strongly extensional for all $a\in S$. 

The following result is proved just as in the classical case.

\begin{lma} \label{groupcondition}
A semigroup $S$ is a group if and only if the maps $\rho_a$ and $\lambda_a$ are surjective for all $a\in S$. 
\end{lma}

For a function $f:X\to Y$, the \emph{kernel} and the \emph{cokernel} are the relations
  \begin{align*}
    \ker f &= \{(x,y)\in X\times X \mid f(x)=f(y)\} \subset X\times X, \\
    \cker f &= \{(x,y)\in X\times X \mid f(x)\apt f(y)\} \subset X\times X. 
  \end{align*}

We generally use the prefix \emph{co-} for the constructive friends of classical concepts, such as co-congruence, co-subsemigroup, cotransitive, etc. Unavoidably, this leads to some clashes with standard terminology; for example, notions like co-ideal and cokernel have other meanings in classical mathematics. It is also worth to point out that, unlike in the classical case, the constructive friends are not dual to their commonly defined counterparts: while the complement of, for example, a coequivalence is an equivalence, the complement of an equivalence need not be a coequivalence in general.

%%%%%%%%%%%%%%%%%%%%%%%%%%%%%%%%%%%%%%%%%%%%%%%%%%%%%%%%%%%%%%%%%%%%%%%%%%%%%%%%
\section{Preliminaries on sets and relations} \label{preliminaries}

%%%%%%%%%%%%%%%%%%%%%%%%%%%%%%%%%%%%%%%%
\subsection{Apartness topology} \label{atopology}

The strongly extensional subsets of a set $X$ with apartness $\apt$ form a topology $\Omega_{\apt}=\Omega_{\apt}^X$ on $X$,
called the \emph{apartness topology}; cf.{} \cite[Section~1.0]{waaldijk96}.
A function $f:X\to Y$ is strongly extensional if and only if it is continuous with respect to the apartness topologies on $X$ and $Y$ \cite[Corollary~2.3.5]{bv11}.
It is straightforward to prove that a subset $A\subset X$ is strongly extensional if and only if the apartness topology on $A$ coincides with the subspace topology induced from the apartness topology on $X$.

We say that a subset $A$ of a topological space $X$ is \emph{closed} if 
$\tild\,A\subset X$ is open ({cf.} \cite[p~23]{vickers89}).
Note that while finite unions of closed sets are closed, it is not possible to prove in general that the intersection of two closed sets is closed; see Example~\ref{notclosed} below.
A topology $\mathcal{T}$ on $X$ is \emph{Fr\'echet}, or $\mathrm{T}_1$, if, whenever $x\apt y$,
there exists a $U\in\mathcal{T}$ such that $x\in U$ and $y\in \tild\, U$.
It is \emph{Hausdorff}, or  $\mathrm{T}_2$, if $x\apt y$ implies the existence of
$U,V\in\mathcal{T}$ such that $x\in U$,  $y\in V$ and $U\cap V=\emptyset$. 
The apartness topology on $X$ is Fr\'echet. 
Frank Waaldijk's thesis \cite[2.0.3, p~73]{waaldijk96} contains an example (within the framework of Brouwer's intuitionistic mathematics) showing that the apartness topology need not be Hausdorff.

\begin{prop} \label{separation}
  Let $X=(X,\mathcal{T})$ be a topological space.
  \begin{enumerate}
  \item The space $X$ is Fr\'echet if and only if $\{x\}\subset X$ is closed for all $x\in X$. \label{separation1}
  \item The space $X$ is Hausdorff if and only if $(\apt) \subset X\times X$ is open in the product topology. \label{separation2}
  \item 
    For $X$ and $Y$ with apartness topology, the apartness topology on $X\times Y$ refines the product topology.
    If the product topology on $X\times X$ coincides with the apartness topology, then $(X,\Omega_{\apt})$ is Hausdorff. \label{separation3}
  \end{enumerate}
\end{prop}

Note that $(\apt) = \tild\,(=)=\tild\,\{(x,x)\in X\times X\}\subset X\times X$, so (2) is the classical result that $X$ is Hausdorff if and only if the diagonal is closed. 

\begin{proof}
The assertions (1) and (2) are proved just as in the classical case. 

For (3), observe that if $U\subset X$ and $V\subset Y$ are strongly extensional, then
$U\times V\subset X\times Y$ is strongly extensional. Since these subsets form a basis of
the product topology, it follows that the apartness topology refines the product
topology.

If every strongly extensional subset of $X\times X$ is open in the product
topology then in particular $(\apt)\subset X\times X$ is open so, by (2), $X$ is Hausdorff. \end{proof}

\begin{lma} \label{senegcompl} 
Let $A\subset X$ be a strongly extensional subset, and $B\subset X$. Then
  \begin{enumerate}
  \item $\tild A=\neg A$; \label{senegcompl1}
  \item $A\subset \neg B \;\equ\; A\subset \tild B$. \label{senegcompl2}
  \end{enumerate}
\end{lma}

\begin{proof}
\eqref{senegcompl1}
The inclusion $\tild A\subset \neg A$ is clear.
For the converse, let $x\in\neg A$. 
By strong extensionality, $\forall_{y,a}\left(a\in A\impl y\in A \eller y\apt  a\right)$. In
particular, substituting $x$ for $y$ we get $\forall_a(a\in A\impl x\in A \eller x\apt a)$,
and since, by assumption, $\neg(x\in A)$, it follows that $x\apt a$ 
for all $a\in A$. Hence $x\in\tild A$.  

\eqref{senegcompl2}\;
$A\subset \neg B \;\equ\; B\subset \neg A \;\stackrel{\eqref{senegcompl1}}{\equ}\; B\subset \tild A \;\equ\; A\subset \tild B$.
\end{proof}

\begin{rmk}
  For a subset $A\subset X$, the statement $\neg A = \tild A$ is equivalent to
  \[\forall_{a\in A,x\in X}(x\notin A \:\impl\: x\apt a) \,,\]
  while $A$ being strongly extensional is the statement
    \[\forall_{a\in A,x\in X}(x\in A \eller x\apt a) \,.\]
    Hence, the converse of Lemma~\ref{senegcompl}\eqref{senegcompl1} cannot be proved in general.
    Indeed, let $P$ be a proposition, $X_P=\{0,1\}/P$, and $a\apt b \equ (a\ne b \och \neg P)$. Then the subset $A=\{1\}\subset X_P$ satisifes $\neg A = \tild A$, whereas strong extensionality is equivalent to $P\eller \neg P$. Thus, the converse of Lemma~\ref{senegcompl}\eqref{senegcompl1} entails the law of excluded middle.
\end{rmk}

The following example shows that the intersection of two closed sets need not be closed.

\begin{ex} \label{notclosed}
  Let $P$ be a proposition, and $\kappa$ the relation on $\{0,1\}$ defined by $x\kappa y \:\equ\: (x\ne y \och P)$. It is straightforward to verify that $\kappa$ is a (tight) apartness relation on the set $X=\{0,1\}/(\neg\kappa)$.

  Setting $A=\{[0]\}\subset X$ and $B = \{x\in X\mid P \}\subset X$, we have $\tild A = \{x\in X\mid (x=[1])\och P \}$, and $\tild B = \{x\in X \mid \neg P \}$, both of which are strongly extensional subsets of $X$. Hence, $A$ and $B$ are closed in the apartness topology.
  On the other hand,
  \begin{align*}
    \tild\,(A\cap B) &= \tild\,\{x\in X\mid (x=[0])\och P\} \\ &=
  \{ x\in X \mid \forall_y(\,(y=[0]\och P) \impl x\kappa y\,) \} 
  =\{x\in X \mid P \impl x=[1]\} 
  \end{align*}
  In particular, $[1]\in \tild\,(A\cap B)$. 
  Assume that $\tild\,(A\cap B)\subset X$ is strongly extensional, that is, that $A\cap B$ is closed. Then either $[0]\in\tild\,(A\cap B)$ or $[0]\kappa [1]$. The former condition is equivalent to $\neg P$, the latter to $P$. Hence, $A\cap B\subset X$ is closed (if and) only if $P$ is decidable.
\end{ex}

%%%%%%%%%%%%%%%%%%%%%%%%%%%%%%%%%%%%%%%%
\subsection{Relations}

\begin{lma} \label{reflexiveconsistent}
 Let $\a$ be a relation on $X$. Then
 \begin{enumerate}
 \item $\tild\a$ is reflexive if and only if $\a$ is strongly irreflexive;
   \label{reflexiveconsistent1}
 \item If $\a$ is reflexive then $\tild\a$ is strongly irreflexive;
   \label{reflexiveconsistent2}
 \item $\tild\tild\,(\apt ) = (\apt )$.
   \label{reflexiveconsistent3}
\end{enumerate}
\end{lma}
 
\begin{proof}
\eqref{reflexiveconsistent1}
$\tild\a$ is reflexive $\equ$
\begin{equation*}
  \forall_a \, a(\tild\a)a  \:\equ\: \forall_{a,x,y}(x\a y \:\impl\: (x,y)\apt(a,a) ) \:\equ\: \forall_{x,y}(x\a y \:\impl\: x\apt y) \:\equ\: \a\subset(\apt)
\end{equation*}
$\equ$ $\a$ is strongly irreflexive.
  
\eqref{reflexiveconsistent2} 
Assume that $(x,y)\in \tild\a$. Then $(x,y)\apt (a,b)$ for all $(a,b)\in\a$, in particular,
since $\tild\a$ is reflexive, $(x,y)\apt (x,x)$. So $x\apt  y$, that is, $\tild\a$ is strongly irreflexive.

\eqref{reflexiveconsistent3}  
The inclusion $(\apt )\subset\tild\tild\,(\apt )$ is clear. On the other hand, $\apt $ is strongly irreflexive by construction, and thus $\tild\,(\apt )$ is reflexive by (1).
Now (2) implies that $\tild\tild\,(\apt )$ is strongly irreflexive, \ie,
$\tild\tild\,(\apt )\subset(\apt )$.
\end{proof}
  
\begin{prop} \label{kappaisse}
  Every co-quasiorder $\kappa$ is a strongly extensional subset of $X\times X$.
\end{prop}

\begin{proof}
For all $(a,b)\in\kappa$, and all $x,y\in X$,
\begin{align*}
   & a\kappa x \eller x\kappa b &
  \impl\quad & a\kappa x \eller x\kappa y \eller y\kappa b &
  \stackrel{\kappa\subset(\apt) }{\impl}\quad & a\apt  x \eller x\kappa y \eller y\apt  b &
  \impl\quad &(a,b)\apt  (x,y) \eller x\kappa y
\end{align*}
that is, $\kappa$ is strongly extensional.
\end{proof}

The following is an immediate consequence of Lemma~\ref{senegcompl}\eqref{senegcompl1} and Proposition~\ref{kappaisse}.

\begin{cor} \label{kappanegcompl}
The identity $(\neg\kappa) = (\tild\kappa)$ holds for any co-quasiorder $\kappa$.
\end{cor}

\begin{prop} \label{apartnessonquotient}
  Let $\e$ be an equivalence, and $\kappa$ a coequivalence on $X$. Then $\kappa$ is
  extensional on the factor set $X/\e$ if and only if $\e\cap\kappa = \emptyset$.
\end{prop}

\begin{proof}
``$\impl$'':
Let $x,y\in X$, and assume that $x\e y$. Extensionality of $\kappa$ on $X/\e$
implies that $x\kappa z \;\equ\; y\kappa z$ for all $z\in X$. In
particular, $x\kappa y \;\equ\; y\kappa y$, so $\neg(x\kappa y)$. 
Hence, $\e\cap\kappa =\emptyset$.

``$\Leftarrow$'': Let $x\e x'$ and $y\e y'$, and assume that $x\kappa y$. Since
$\neg(x\kappa x')$ and $\neg(y'\kappa y)$, we have
$$x\kappa y \;\impl\; x\kappa x' \eller x'\kappa y \;\impl\;
x\kappa x' \eller x'\kappa y' \eller y'\kappa y  \;\impl\; x'\kappa y' \,,$$
which proves extensionality of $\kappa$ on $X/\e$.
\end{proof}

\begin{rmk}
An equivalent formulation of the result in Proposition~\ref{apartnessonquotient} is that $\kappa$ is an apartness relation on $X/\epsilon$ if and only if $\e\cap\kappa=\emptyset$.
In particular, any coequivalence $\kappa$ is an apartness relation on $X/(\neg\kappa)$.
\end{rmk}

\begin{prop} \label{serelations}
  For any relation $\alpha\subset X\times Y$, the following statements are
  equivalent:
  \begin{enumerate}
  \item $\a$ is a strongly extensional subset of $X\times Y$;
  \item the subsets $x\alpha\subset Y$ and $\a y\subset X$ are strongly extensional for all
    $x\in X$, $y\in Y$;
  \item the subsets $A\alpha\subset Y$ and $\a B\subset X$ are strongly extensional for all
    subsets $A\subset X$, $B\subset Y$.
  \end{enumerate}
\end{prop}

\begin{proof}
The implication (3)$\,\impl\,$(2) is trivial. 
Since the strongly extensional subsets of $X$ form a topology, and $A\a=\bigcup\{x\a \mid x\in A \}$, we have (2)$\,\impl\,$(3).

(1)$\,\impl\,$(2):
The condition that $\a\subset X\times Y$ is a strongly extensional subset means that for
all $(x,y),(a,b)\in X\times Y$ such that $a\a b$, either $x\a y$ or $(x,y)\apt(a,b)$
holds. 
Specifying $a=x$ then gives that for all $y\in Y$ and all $b\in x\a$, either $y\in x\a$ or
$b\apt y$. That is, $x\a$ is a strongly extensional subset of $Y$. 
Similarly, one shows that $\a y\subset X$ is strongly extensional.

For (2)$\,\impl\,$(1), assume that $x\a\subset Y$ and $\a y\subset X$ are strongly
extensional for all $x\in X$, $y\in Y$, and let $(a,b)\in\a$, $(c,d)\in X\times Y$.
Because $a\a\subset Y$ is strongly extensional, we have that either $a\a d$
or $b\apt d$. If $a\a d$, then $\a d\subset X$ being strongly extensional implies that
either $c\a d$ or $c\apt a$. 
Taken together, this means that
\[\forall_{(a,b)\in\a,\:(c,d)\in X\times Y} (c\a d \eller a\apt c \eller b\apt d),\] 
that is, $\a\subset X\times Y$ is strongly extensional. 
\end{proof}

\begin{cor}
  If $\a\subset X\times Y$ and $\beta\subset Y\times Z$ are strongly extensional, then
  $\a\circ\b\subset X\times Z$ is strongly extensional. 
\end{cor}

\begin{proof}
Let $z\in Z$. Then
\[ (\a\circ\b)z = \{x\in X \mid \exists_{y\in Y}(y\in\b z \,\wedge\, x\in\a y) \} =
\a(\b z) \,.
\]
By Proposition~\ref{serelations}:(3)$\impl$(1), $\a B\subset X$ is strongly extensional for
any subset $B\subset Y$ so, in particular, $(\a\circ\b)z = \a(\b z)\subset X$ is strongly
extensional. Similarly, the subset $x(\a\circ\b)$ of $Z$ is strongly extensional for all
$x\in X$. With the implication (2)$\impl$(1) of  Proposition~\ref{serelations}, we
conclude that $\a\circ \b\subset X\times Z$ is strongly extensional. 
\end{proof}

%%%%%%%%%%%%%%%%%%%%%%%%%%%%%%%%%%%%%%%%
\subsection{Coarse equality}

Below, we shall see that it is always possible to pass to a context of tight apartness, by re-defining equality as the negation of the apartness relation.

\begin{dfn}
  The \emph{coarse equality} relation ``$\approx$'' on $X$ is the relation
  $(\approx)=(\neg\apt )$. 
\end{dfn}

Since $\apt $ is a coequivalence, $\approx$ is an equivalence. From Lemma~\ref{reflexiveconsistent}\eqref{reflexiveconsistent3} and  Corollary~\ref{kappanegcompl}, we get the identities $\tild\,(\approx) = (\apt)$ respectively $(\approx)=\tild\,(\apt)$.

\begin{prop} \label{relapart}
  For any relation $\a$ on $X$,
  \begin{enumerate} 
  \item $\a$ is strongly irreflexive if and only if $(\approx)\subset\tild\a$;
     \label{relapart1}
   \item
     $\a\subset(\approx) \;\equ\; \a\cap(\apt )=\emptyset \;\equ\; (\apt)\subset \tild\a$.
     \label{relapart2}
   \item If $\a$ is reflexive then
    $\a\subset(\approx) \;\equ\; (\apt) = \tild\a$.
     \label{relapart3}
  \end{enumerate}
\end{prop}

\begin{proof}
\eqref{relapart1}
By definition, $\a$ is strongly irreflexive if and only if $\a\subset(\apt) = \tild\,(\approx)$ which, by Remark~\ref{apartrmk}\eqref{apartrmk3}, is equivalent to  $(\approx)\subset \tild\a$.

\eqref{relapart2}
The first equivalence is immediate from the fact that $(\approx) = \neg(\apt)$, and the second follows from Remark~\ref{apartrmk}\eqref{apartrmk3}.

\eqref{relapart3}
If $\a$ is reflexive then $\tild\a\subset(\apt)$ by Lemma~\ref{reflexiveconsistent}\eqref{reflexiveconsistent2}.
\end{proof}

A constructive version of the first isomorphism theorem for semigroups with apartness was first given in \cite[Theorem~2.5]{cmr13}.
We shall need the following, somewhat generalised version.

\begin{prop} \label{basicfactor}
  Let $f:X\to Y$ be map between sets with apartness $X$ and $Y$, $K=\ker f$, and let
  $\zeta$ be a coequivalence on $X$ such that $\zeta\cap K=\emptyset$. Then:
  \begin{enumerate}
  \item $X/K$ is a set with apartness induced from $\zeta$; \label{basicfactor1}
  \item the projection map $\pi: X\to X/K,\:x\mapsto [x]$ is strongly extensional and
    surjective. \label{basicfactor2}
  \item the map $f$ induces an injective map $\theta:X/K\to Y$ given by $\theta([x])=f(x)$, and $f=\theta\pi$; \label{basicfactor3}
  \item the map $\theta$ is strongly extensional if and only if $\cker f\subset\zeta$, and apartness injective if and only if $\zeta\subset\cker f$. \label{basicfactor4}
  \end{enumerate}
  Moreover, if $X$ is a semigroup with apartness and $\zeta$ a co-congruence, then $(X/K,\zeta)$ is a semigrop with apartness, and $\pi$ a morphism in $\aSg$. If, in addition $Y$ is a semigroup with apartness and $f$ a morphism, then $\theta$ is also a morphism.
\end{prop}

\begin{proof}
\eqref{basicfactor1}
This was proved in Proposition~\ref{apartnessonquotient}.

\eqref{basicfactor2}
The projection map $\pi:X\to X/K$ is surjective by definition. Strong extensionality
follows from the strong irreflexivity of $\zeta$: $\pi(x)\apt \pi(y)\;\equ\; x\zeta y \;\impl\; x\apt  y$.

\eqref{basicfactor3}
This was shown in \cite{cmr13}.

\eqref{basicfactor4}
Apartness injectivity of $\theta$ means, by definition, that $\theta(x)\apt \theta(y)$
whenever $[x]\apt [y]$.
Now
\[\theta(x)\apt \theta(y) \;\equ\;f(x)\apt  f(y) \;\equ\;(x,y)\in\cker f 
\quad\mbox{and} \quad [x]\apt  [y] \;\equ\; (x,y)\in\zeta \]
so $\theta$ is injective if and only if $\zeta\subset\cker f$.
Similarly, $\cker f\subset \zeta$ is precisely the condition that 
$ \theta(x)\apt \theta(y) \;\impl\; [x] \apt [y]$, that is, that $\theta$ is strongly
extensional.

Assume that $X$ is a semigroup with apartness and $\zeta$ a co-congruence.
Clearly, multiplication in $X/K$ is well-defined and associative. Since $\zeta$ is
co-compatible with multiplication in $X$, we have
\[[x][y]\apt [z][w] \;\equ\; (xy)\zeta(zw) \;\impl\; x\zeta z \eller y\zeta w
\;\equ\; 
[x]\apt  [z] \eller [y]\apt  [w] \,,\]
which means that the multiplication in $X/K$ is strongly extensional.

The remaining statements are part of classical theory.
\end{proof}

\begin{rmk}
  The map $f$ in the proposition above is strongly extensional if and only if $\cker f$ is strongly irreflexive. By (4), if $\theta$ is strongly extensional then $\cker f\subset\zeta\subset(\apt )\subset X\times X$, implying that $f$ is strongly extensional as well.
\end{rmk}

\begin{thm} \label{tightfactor}
  Let $X$ and $Y$ be sets with apartness. Then the following hold:
  \begin{enumerate}
  \item The apartness on $X/{\approx}$ induced from $X$ is tight. \label{tightfactor1}
  \item Any strongly extensional function $f:X\to Y$ induces a strongly extensional
    function $\bar{f}:X/{\approx} \to Y/{\approx}$, defined by $\bar{f}([x])=[f(x)]$. \label{tightfactor2}
  \item The canonical map 
    \[t : (X\times Y)/{\approx} \to (X/{\approx})\times(Y/{\approx}),\:
    [(x,y)]\mapsto ([x],[y]) \]
    is an isomorphism in $\aSets$. \label{tightfactor3}
  \item If $S$ is a semigroup with apartness then $S/{\approx}$ is a semigroup with tight apartness, with apartness relation and multiplication induced from $S$. \label{tightfactor4}
  \end{enumerate}
\end{thm}

Denote by $\Sg^{!}$ the full subcategory of $\aSg$ consisting of semigroups with tight apartness.

\begin{cor}
The assignments $S\mapsto S/{\approx}$, $f\mapsto\bar{f}$ define a functor $\aSg\to\Sg^{!}$, which is left inverse to the inclusion functor $\Sg^{!}\to\aSg$. 
\end{cor}

\begin{proof}[Proof of Theorem~\ref{tightfactor}]
\eqref{tightfactor1}
Let $x,y\in X$. By definition, $x\approx y \;\equ\; \neg(x\apt  y)$, meaning that $\apt $ is tight on $X/{\approx}$.

\eqref{tightfactor2}
Denote by $\pi_X:X\to X/{\approx}$ and $\pi_Y:Y\to Y/{\approx}$ the quotient projections. The map $f$ being strongly extensional means that $f(x)\apt  f(y)\;\impl\;x\apt  y$ for all $x,y\in X$. Thus $x\approx y \;\impl\; f(x)\approx f(y)$, which proves that $f$ factors through $\pi_X:X\to X/{\approx}$; in other words: there exists a map $\theta:X/{\approx}\to Y$ such that $f=\theta\pi_X$. 
Since $\pi_X(x)\apt \pi_X(y)\;\equ x\apt  y$, it follows that $\theta$ is strongly
extensional. 
Set $\bar{f}=\pi_Y\theta:X/{\approx}\to Y/{\approx}$. By definition, 
$\bar{f}([x])=\pi_Y\theta([x]) = \pi_Yf(x)=[f(x)]$. Since both $\theta$ and $\pi_Y$ are
strongly extensional, so is $\bar{f}$. 

\eqref{tightfactor3}
Set $f:X\times Y\to X/{\approx}\times Y/{\approx}, \: (x,y)\mapsto ([x],[y])$. By
Proposition~\ref{basicfactor}(3), $f$ induces an injective map 
$t=\theta:(X\times Y)/(\ker f)\to (X/{\approx})\times (Y/{\approx})$, which is surjective since $f$ is surjective. 
Now 
\begin{multline*}
  f(x,y)=f(z,w)\;\equ\; ([x],[y]) = ([z],[w]) \;\equ\; x\approx z \;\wedge\;
  y\approx w  \\
\;\equ\; \neg (x\apt z \eller y\apt w)
\;\equ\; (x,y)\approx(z,w) \;\equ\; [(x,y)]=[(z,w)]
\end{multline*}
that is, $\ker f = (\approx)\subset (X\times Y)^2$.
Moreover,
\[ \cker f = 
\{((x,y),(z,w))\in(X\times Y)^2 \mid ([x],[y])\apt ([z],[w])\}= (\apt )\subset (X\times Y)^2
\] 
so, by Proposition~\ref{basicfactor}\eqref{basicfactor4}, the map 
$t:(X\times Y)/{\approx}\to (X/{\approx})\times (Y/{\approx})$ is injective and strongly extensional. Thus, is is an isomorphism in $\aSets$.

\eqref{tightfactor4}
The multiplication map $X\times X\to X$ is strongly extensional and so, by \eqref{tightfactor2}, it induces a strongly extensional map $\bar{\mu}:(X\!\times\! X)/{\approx} \to X/{\approx}$. Now
\[ \bar{\mu}t\inv:(X/{\approx})\!\times\!(X/{\approx}) \to X/{\approx},\:([x],[y])\mapsto[xy] \]
gives the semigroup structure on $X/{\approx}$. It is clearly associative, and strongly
extensional since $t\inv$ and $\bar{\mu}$ are strongly extensional.
\end{proof}

\subsection{Filled product and the constant domains property} \label{filledcd}

The filled product, defined below, is a constructive friend of the composition ``$\circ$''. The concept originates in work by Romano, see for example \cite{romano02}.

\begin{dfn}
  Let $\a$ and $\b$ be relations on $X$. The \emph{filled product} of $\a$ and $\b$ is the relation
  \[ \a*\b = \left\{(x,y)\in X\times X \mid \forall_{z\in X}(x\a z \eller z\b y) \right\} \subset X\times X. \]
\end{dfn}

The following are some basic properties of the filled product. 

\begin{lma} \label{*properties}
  Let $\a$ and $\b$ be relations on $X$.
  \begin{enumerate}
      \item The inclusion $(\neg\a) * (\neg\b)\subset\neg(\a\circ\b)$ holds.
    \label{*properties1}
  \item The identity $(\a*\b)\inv = \b\inv*\a\inv$ holds. In particular, if $\a$ is symmetric then $\a*\a$ is symmetric.
    \label{*properties2}
      \item
    If $\forall_{x \in X}\exists_{w\in X} \neg(x\a w)$, then $\a*\b\subset \b$.
    If $\forall_{y \in X}\exists_{w\in X} \neg(w\b y)$, then $\a*\b\subset \a$.
    \label{*properties3}
  \item The relation $\a$ is cotransitive if and only if $\a\subset\a*\a$.
    \label{*properties4}
  \item If $\a$ is strongly irreflexive, then it is a co-quasiorder if and only if $\a=\a*\a$.
    \label{*properties5}
  \item If $\a\subset\gamma$ and $\b\subset\d$ then $\a*\b\subset\gamma*\d$.
    \label{*properties6}
      \item If $\kappa\subset\a\cap\b$ is a cotransitive relation, then $\kappa\subset \a*\b$.
    \label{*properties7}
  \end{enumerate}
\end{lma}

\begin{proof}
\eqref{*properties1}
For all $x,y\in X$,
\[(x,y)\in (\neg\a)*(\neg\b) \:\equ\: \forall_z(\neg(x\a z)\eller \neg(z\b y))
\:\impl\: \neg\exists_z(x\a z \,\wedge\, z\b y) \:\equ\: \neg(x(\a\circ\b)y) \,,\] 
hence, $(\neg\a) * (\neg\b)\subset\neg(\a\circ\b)$.

\eqref{*properties2} For all $x,y\in X$,
\begin{align*}
 (x,y)\in(\a*\b)\inv \;&\equ\; (y,x)\in\a*\b
  \;\equ\; \forall_z\left( y\a z \eller z\b x\right) \\
  \;&\equ\; \forall_z\left( z\a\inv y \eller x\b\inv z\right) 
  \;\equ\; (x,y)\in\b\inv*\a\inv\,.
\end{align*}

\eqref{*properties3}
If $\a$ is irreflexive then
$$(x,y)\in\a*\b \;\equ\; \forall_z\left( x\a z \eller z\b y\right) \;\stackrel{z=x}{\impl}\;
(x\a x \eller x\b y) \;\impl\; (x,y)\in\b \,.$$
The case when $\b$ is irreflexive is analogous. 

\eqref{*properties4}
$\displaystyle \a \mbox{ is cotransitive} \;\equ\;$
\[\forall_{x,y} \left(x\a y \;\impl\; \forall_z(x\a z \eller z\a y) \right) \;\equ\; 
\forall_{x,y} \left(x\a y \;\impl\; (x,y)\in\a*\a \right) \;\equ\; \a\subset\a*\a\,.\]

\eqref{*properties5}
This is immediate from \eqref{*properties2} and \eqref{*properties3}.

\eqref{*properties6}
$\displaystyle (x,y)\in \a*\b \;\equ\; \forall_z(x\a z \eller z\b y)
\;\stackrel{a\subset\gamma,\,\b\subset\d}{\implies}\;  
\forall_z(x\gamma z \eller z\d y) \;\equ\; (x,y)\in\gamma*\d$.

\eqref{*properties7}
From \eqref{*properties4} and \eqref{*properties6}, it follows that $\kappa\subset\kappa*\kappa \subset\a*\b$.
\end{proof}

\begin{rmk}
  With respect to Lemma~\ref{*properties}\eqref{*properties1}, it is easy to prove that if $\a$ and $\b$ are detachable subsets of $X\times X$, then $(\neg\a)*(\neg\b) = \neg(\a\circ\b)$. However, in the general case, it is not possible to prove equality.
  For example, given a proposition $P$, define relations $\a$ and $\b$ on the set $X=\{0\}$ by $0\a 0 \:\equ\: P$ and $\b= \neg\a$. Then $(0,0)\in\neg(\a\circ\b)$ is equivalent to $\neg(P\wedge\neg P)$ and hence always true, whilst $(0,0)\in(\neg\a)*(\neg\b) \:\equ\: \neg P \vee \neg\neg P$, which is not constructively provable in general.
\end{rmk}

We say that a set $X$ has the constant domains property if
\begin{equation} \label{cdaxiom}
  \forall x(P(x) \eller Q) \,\impl\, (\forall x\,P(x)) \eller Q 
\end{equation}
holds for all predicates $P(x)$ and $Q$, with $x$ not free in $Q$. Note that the converse implication is always constructively valid.
It is straightforward to see that every finite set has the constant domains property, and that this property is preserved by surjective images: if $f:X\to Y$ is a surjective function, and $X$ has the constant domains property, then so does $Y$.  In particular, the following holds.

\begin{lma} \label{fecd}
  Every finitely enumerable set has the constant domains property.
\end{lma}

By contrast, the constant domains property cannot be proved to be preserved by subsets. For example, let $P$ be a proposition, $X=\{0\}$, $A=\{x\in X \mid P\}\subset X$, and $Q(x) \Leftrightarrow \bot$. Then $\forall_{x\in A}( Q(x)\eller P)$ is true, whereas \[(\forall_{x\in A}\,Q(x)) \eller P \:\equ\: ((P\impl Q(0))\eller P) \:\equ\: (\neg P \eller P) \,.\]
So while $X$ has the constant domains property, its validity in the subset $A\subset X$ entails LEM.

The following result, connecting the constant domains property with associativity of the filled produt, will be of importance for us later.

\begin{prop} \label{cd*assoc}
  A set $X$ has the constant domains property if and only if $\a*(\b*\gamma) = (\a*\b)*\gamma$ holds for all relations $\a,\b,\gamma$ on $X$.
\end{prop}

\begin{proof}
Let $\a\subset X\times Y$, $\b\subset Y\times Z$ and $\gamma\subset Z\times W$ be relations. Then $\a*(\b*\gamma) = (\a*\b)*\gamma$ means, by definition, that the following equivalence holds for all $x\in X$ and $w\in W$:
\begin{equation} \label{*assoc}
\forall_{y\in Y} \left(\, x\a y \eller \forall_{z\in Z}(y\b z \eller z\gamma w) \,\right)
\;\equ\;
\forall_{z\in Z} \left(\, \forall_{y\in Z}(x\a y \eller y\b z) \eller z\a w \,\right)
\end{equation}
Assuming that the constant domains principle \eqref{cdaxiom} holds over $Y$ and $Z$, each side of \eqref{*assoc} becomes equivalent to the proposition
\[ \forall_{y\in Y,z\in Z} \left(\, x\a y \eller y\b z \eller z\gamma w \,\right)
\]
and thus, this principle entails the associativity of the filled product. Indeed, the converse is true as well:
if $\a*(\b*\gamma) = (\a*\b)*\gamma$ for all relations $\a$, $\b$ and $\gamma$ on a set $X$, then \eqref{cdaxiom} holds over $X$.
To see this, let $P(x)$ and $Q$ be formulae, with $x$ not free in $Q$, and define 
\[ x\a y \:\equ\: \bot,\quad x\b y \:\equ\: P(x), \quad\mbox{and}\quad
x\gamma y \:\equ\: Q \,. \]
Then, for any elements $a,b\in X$:
\begin{align*}
  a(\a*(\b*\gamma))b \:&\equ\:
  \forall_{x}(\,\bot \eller\forall_{y}(P(x)\eller Q) \,) \:\equ\:
  \forall_{x}(P(x)\eller Q \,) \,, \;\mbox{and}\\
  a((\a*\b)*\gamma)b \:&\equ\:
  \forall_{y}(\,\forall_{x}(\bot \eller P(x))\eller Q) \,) \:\equ\:
  \forall_{x}P(x)\eller Q \,.
\end{align*}
In particular, if $\a*(\b*\gamma) = (\a*\b)*\gamma$ then \eqref{cdaxiom} holds in $X$. 
\end{proof}

\begin{rmk}
  With a somewhat more elaborate example, one can prove that even the (much) weaker condition of \emph{third-power associativity} $\a*\a^{*2}=\a^{*2}*\a $ of the filled product entails the constant domains property.
\end{rmk}

%%%%%%%%%%%%%%%%%%%%%%%%%%%%%%%%%%%%%%%%%%%%%%%%%%%%%%%%%%%%%%%%%%%%%%%%%%%%%%%%
\section{Kernels} \label{kernels}

In this section, we shall introduce constructive friends of the transitive closure and the congruence closure of a relation $\a$: the cotransitive kernel and the co-congruence kernel, respectively. The cotransitive kernel -- when it exists -- is the maximal cotransitive subrelation of $\a$ and, similarly, the co-congruence kernel is the maximal co-congruence contained in $\a$.

One of the motivating problems is to characterise when a given equivalence (or congruence) $\e$ on $X$ can be obtained as the negation of a coequivalence (co-congruence) and thus, when there exists a tight apartness on $X/\e$.

Recall that a \emph{kernel operator} on a partially ordered set $(P,\le)$ is a function $k:P\to P$ satisfying
\begin{enumerate} \renewcommand{\labelenumi}{\roman{enumi})}
\item $k(x)\le x$;
\item $x\le y \:\impl\: k(x)\le k(y)$;
\item $k^2 = k$,
\end{enumerate}
for all $x,y\in P$.
A \emph{closure operator} on $P$ is the same thing as a kernel operator on the opposite poset $(P, (\le)^{-1})$. 
  
We start with a brief look at the classical concepts of transitive closure and congruence closure, from a constructive viewpoint.

%%%%%%%%%%%%%%%%%%%%%%%%%%%%%%%%%%%%%%%%
\subsection{Transitive closure and congruence closure}

These notions are essentially constructive, and present no real difficulties in our setting. Below, we spell out the basic results -- the proofs are the same as in the classical setting (as presented, for example, in Howie's book \cite{howie95}). 

It is easy to see that composition ``$\circ$'' of relations is associative; in particular, it follows that $\a^{\circ m}\circ\a^{\circ n} = \a^{\circ (m+n)}$ for all $\a\subset X\times X$, $m,n\ge1$.

\begin{dfn}
  For any relation $\a$ on $X$, define $\a^\infty = \bigcup_{n\ge
    1}\a^{\circ n}$, and  
  $\a^e = (\a\cup \a\inv\cup (=))^\infty$. 
  The relation $\a^\infty$ is called the \emph{transitive closure} of $\a$. 
\end{dfn}

\begin{prop}
  Let $\mathcal{X}$ be a set of relations on $X$, and $\a,\b\in\mathcal{X}$. 
  \begin{enumerate}
  \item
The functions $\mathcal{X}\cup\mathcal{X}^\infty\to \mathcal{X}\cup\mathcal{X}^\infty,\;\a\mapsto \a^\infty$, and
$\mathcal{X}\cup\mathcal{X}^e\to \mathcal{X}\cup\mathcal{X}^e,\; \a\mapsto \a^e$, are closure operators.
  \item $\a^\infty = \a$ if and only if $\a$ is transitive, and $\a^e = \a$ if and
    only if $\a$ is an equivalence relation.
  \item If $\a$ and $\b$ are equivalence relations satisfying $\a\circ\b=\b\circ\a$, then $(\a\cup\b)^{\infty} = \a\circ\b$. \label{viso}
  \end{enumerate}
\end{prop}

\begin{dfn}
  For any relation $\a$ on a semigroup $S$, define $\a^c$ by
\[ x\a^cy \:\equ\: 
x\a y \eller\exists_{s,t,u,v\in S} (\, u\a v \,\wedge\, (x,y)=(sut,svt) \,) \,. \]
\end{dfn}

\begin{prop} \label{compatibleclosure}
    Let $\mathcal{X}$ be a set of relations on $X$, and $\a,\b\in\mathcal{X}$.  
    \begin{enumerate}
    \item The function $\mathcal{X}\cup\mathcal{X}^c\to \mathcal{X}\cup\mathcal{X}^c$, given by $\a\mapsto \a^c$, is a closure operator.
    \item $\a^c=\a$ if and only if $\a$ is left and right compatible with the multiplication
      in $S$.
    \item $(\a^c)\inv = (\a\inv)^c$
    \item $(\a\cup\b)^c = \a^c\cup \b^c$. \label{compatibleunion}
    \end{enumerate}
\end{prop}

\begin{proof}[Proof of Proposition~\ref{compatibleclosure}\eqref{compatibleunion}]
As disjunctions can be treacherous in constructive mathematics, we give the proof of this
result. 
For any $x,y\in S$,
\begin{align*}
  &x(\a\cup\b)^cy \;\equ\;
  (x\a y \eller x\b y) \eller 
  \exists_{s,t,u,v\in S} \,[\, u(\a\cup\b)v \wedge (x,y)=(sut,svt) \,] \\
  &\;\equ\;
  (x\a y \eller x\b y) \eller 
  \exists_{s,t,u,v\in S} \,[\, (u\a v \eller u\b v) \wedge (x,y)=(sut,svt) \,] \\
  &\;\equ\;
  (x\a y \eller x\b y) \eller 
  \exists_{s,t,u,v\in S} \,[\, (u\a v \wedge (x,y)=(sut,svt)) \eller
    (u\b v \wedge (x,y)=(sut,svt))\,] \\
  &\;\equ\;
  x\a y \eller x\b y \eller 
  \exists_{s,t,u,v\in S} \,[\, (u\a v \wedge (x,y)=(sut,svt)) \,] 
  \eller \exists_{s,t,u,v\in S} \,[\,(u\b v \wedge (x,y)=(sut,svt))\,] \\
  &\;\equ\; x\a^c y \eller x\b^cy \;\equ\; x(\a^c\cup\b^c)y \,.
\end{align*}
\end{proof}

\begin{dfn}
  The \emph{congruence closure} of a relation $\a\subset S\times S$ is the relation
  $\a^k = (\a^c)^e$.
\end{dfn}

\begin{prop} \label{closureprop}
    Let $\mathcal{X}$ be a set of relations on $X$, and $\a,\b\in\mathcal{X}$. 
  \begin{enumerate}
  \item The function $\mathcal{X}\cup\mathcal{X}^k\to \mathcal{X}\cup\mathcal{X}^k$, given by $\a\mapsto \a^k$, is a closure operator.
  \item $\a^k=\a$ if and only if $\a$ is a congruence on $S$. 
  \item If $\a$ and $\b$ are congruences on $S$, then $(\a\cup\b)^k = (\a\cup\b)^\infty$. 
  \end{enumerate}
\end{prop}

%%%%%%%%%%%%%%%%%%%%%%%%%%%%%%%%%%%%%%%%

\subsection{The cotransitive kernel}

The principal aim of this section is to define a constructive analogue of the transitive closure  -- the cotransitive kernel. 
Unfortunately, as we shall see in Section~\ref{notadmissible}, it turns out that the natural candidate for such a relation cannot be shown to be cotransitive in general. We will, however, attain our goal for some classes of relations, including all relations on finitely enumerable sets -- see Section~\ref{finenumsets}.

Given $\a\subset X\times X$, define a set $\<\a\>^*$ of relations on $X$ inductively by $\a\in\<\a\>^*$ and $\b*\gamma\in\<\a\>^*$ whenever $\b,\gamma\in\<\a\>^*$.
That is, $\<\a\>^*$ is the set of relations on $X$ generated by $\a$ under the filled product operation.
We remark that every element $\b$ in $\<\a\>^*$ can be written as $\b = g(\a)$ for some non-associative, non-commutative monomial $g(T)$ in a formal variable $T$ (although not necessarily in a unique way).
Here, $g(\a)$ is defined recursively by
\[ g(\a) =
\begin{cases}
  \a &\mbox{if}\quad g(T) = T,\\
  g_1(\a)*g_2(\a) &\mbox{if}\quad g(T) = g_1(T)g_2(T).
\end{cases}
\]
If filled products of elements in $\<\a\>^*$ are associative, then $\<\a\>^* = \{\a^{*n}\}_{n\ge1}$, where $\a^{*1} = \a$ and $\a^{*(n+1)} = \a*\a^{*n}$.

Set $\tilde{\a} = \bigcap\<\a\>^* = \bigcap_{\beta\in\<\a\>^*}\beta\subset X\times X$.
From Lemma~\ref{*properties}(\ref{*properties4},\ref{*properties6}), it follows that if $\kappa\subset X\times X$ is a cotransitive relation contained in $\a$, then $\kappa\subset \tilde{\a}$. Assuming LEM gives that $\tilde{\a}$ is cotransitive and $\tilde{\a} = \neg((\neg\a)^\infty)$, but these statements are not constructively valid in general.

\begin{dfn}
  \begin{enumerate}
  \item A relation $\a\subset X\times X$ is said to be \emph{(cotransitively) admissible} if $\tilde{\a}= \bigcap\<\a\>^*$ is cotransitive.
  \item If $\a$ is admissible then $\tilde{\a}$ is called the \emph{cotransitive kernel} of $\a$, and denoted by $c(\a)$.  
\item
  A relation $\a\subset X\times X$ is \emph{short} if there exists a positive integer $n$ and $\b_1,\ldots,\b_n\in \<\a\>^*$ such that $\bigcap\<\a\>^*=\b_1\cap\cdots\cap\b_n$.
  \end{enumerate}
\end{dfn}

\begin{rmk}
If $\a\subset X\times X$ is irreflexive, then so is any sub-relation $\b\subset\a$. So if $\b_1\cap\cdots\cap\b_n = \tilde{\a}$ for $\b_1,\ldots,\b_n\in\<\a\>^*$ then, by Lemma~\ref{*properties}\eqref{*properties3}, $\b_1*\cdots*\b_n\subset\b_1\cap\cdots\cap\b_n=\tilde{\a}$. As $\b_1*\cdots*\b_n\in\<\a\>^*$, we get $\tilde{\a}\subset \b_1*\cdots*\b_n\subset\b_1\cap\cdots\cap\b_n=\tilde{\a}$, and thus $\tilde{\a} = \b_1*\cdots*\b_n\in\<\a\>^*$.
This means that an irreflexive relation $\a$ is short if and only if $\tilde{\a}\in\<\a\>^*$.
\end{rmk}

In practice, it may of course be difficult to determine whether or not a given relation is admissible. The concept of a short relation gives us a more workable, sufficient criterion for admissibility.

\begin{prop} \label{atilde-cotrans}
  Let $\a$ be a relation on $X$, such that either
  \begin{enumerate}
  \item $\a$ is short, or \label{atilde-cotrans1}
  \item the set $\<\a\>^*$ has the constant domains property. \label{atilde-cotrans2}
  \end{enumerate}
  Then $\a$ is admissible.
\end{prop}

\begin{proof}
\eqref{atilde-cotrans1}
Assume that $\a$ is short.
For all $x,y\in X$, we have
\begin{align*}
  x(\tilde{\a}*\tilde{\a})y
  \:\equ\:&
   \forall_{z\in X}\left( x\tilde{\a}z \eller z\tilde{\a}y \right) 
   \:\equ\:
   \forall_{z\in X}\left(\bigwedge_{i=1}^n(x\b_i z) \eller \bigwedge_{j=1}^n(z\b_j y)  \right) \\
   \:\equ\:&
   \forall_{z\in X} \left(\bigwedge_{i,j=1}^n (x\b_i z \eller z\b_j y) \right)
   \:\equ\:
   \bigwedge_{i,j=1}^n \left(\forall_{z\in X}(x\b_i z \eller z\b_j y) \right) 
   \:\equ\:
   \bigwedge_{i,j=1}^n x(\b_i*\b_j)y \,,
\end{align*}
that is, $\tilde{\a}*\tilde{a} = \bigcap_{i,j=1}^n(\b_i*\b_j)$.
But as $\b_i*\b_j\in\<\a\>^*$ for all $i$ and $j$, the inclusion $ \tilde{\a} = \bigcap\<\a\>^* \subset \bigcap_{i,j=1}^n(\b_i*\b_j)$ holds, and thus $\tilde{\a}\subset\tilde{\a}*\tilde{\a}$. So the relation $\tilde{\a}$ is cotransitive, by Lemma~\ref{*properties}\eqref{*properties4}.

\eqref{atilde-cotrans2}
Assume that $\<\a\>^*$ has the constant domains property. Then, for all $x,y\in X$,
\begin{align*}
  x(\tilde{\a}*\tilde{\a})y &\;\equ\; \forall_{z\in X}(x\tilde{\a}z \eller z\tilde{\a}y)
  \;\equ\;
  \forall_{z\in X}\left(\forall_{\b\in\<\a\>^*}(x\b z) \eller \forall_{\gamma\in\<\a\>^*}(z\gamma y)\right) \\
  &\stackrel{\rm CD}{\;\equ\;}
  \forall_{\b,\gamma\in\<\a\>^*}  \forall_{z\in X}((x\b z) \eller (z\gamma y))
  \;\equ\;
  \forall_{\b,\gamma\in\<\a\>^*}\,  x(\b*\gamma) y \,.
\end{align*}
On the other hand, if $x\tilde{\a}y$ then $x(\b*\gamma) y$ holds for all $\b,\gamma\in\<\a\>^*$, and thus $x(\tilde{\a}*\tilde{\a})y$. This proves that $\tilde{\a}\subset\tilde{\a}*\tilde{\a}$, whence Lemma~\ref{*properties}\eqref{*properties4} implies that  $\tilde{\a}$ is cotransitive.
\end{proof}

\begin{prop}\label{cproperties}
  Let $\mathcal{X}$ be a set of admissible relations on $X$, and $\a\in\mathcal{X}$.
  \begin{enumerate}
  \item $c$ is a kernel operator on the set $\mathcal{X}\cup c(\mathcal{X})$. 
    \label{cproperties1}
    \item The identity $c(\a) = \a$ holds if and only if $\a$ is cotransitive.
      \label{cproperties2}
  \item If $\a$ is symmetric/irreflexive/strongly irreflexive, then so is $c(\a)$.    \label{cproperties3}
    \item If $\kappa\subset\a$ is a cotransitive relation, then $\kappa\subset c(\a)$. 
  \label{cproperties4}
\end{enumerate}
\end{prop}

\begin{proof}
  \eqref{cproperties1}
  First, we remark that every cotransitive relation $\kappa$ is short, and $c(\kappa)=\kappa$. Thus every relation in the set $\mathcal{X}\cup c(\mathcal{X})$ is indeed admissible, and $c(c(\a)) = c(\a)$.
Since $c(\a)= \bigcap\<\a\>^*$ and $\a\in\<\a\>^*$, we have $c(\a)\subset\a$. Monotonicity follows, by induction, from Lemma~\ref{*properties}\eqref{*properties6}. 

\eqref{cproperties2}
As noted above, $c(\a)= \a$ if $\a$ is cotransitive. Since $c(\a)$ is cotransitive by definition, the converse implications holds as well.

\eqref{cproperties3}
By an induction using Lemma~\ref{*properties}\eqref{*properties2}, we get that $c(\a)$ is symmetric if $\a$ is symmetric. The relation $\a$ being irreflexive means that $\a\subset(\neq)$, and since $c(\a)\subset \a$ by \eqref{cproperties1}, it follows that $c(\a)$ is irreflexive too. Strong irreflexivity follows analogously. 

\eqref{cproperties4}
Applying \eqref{cproperties1} to the set $\mathcal{X}=\{\kappa,\a\}$, the inclusion $\kappa\subset\a$ gives $\kappa = c(\kappa) \subset c(\a)$. 
\end{proof}

Under suitable admissibility assumptions, we can now characterise the transitive relations that occur as negations of cotransitive relations.

\begin{cor} \label{complementcor}
  \begin{enumerate}
  \item Let $\a\subset X\times X$ be a transitive relation such that $\neg\a$ is admissible. The following two statements are equivalent:
    \begin{enumerate}
    \item $\a=\neg \kappa$ for some cotransitive relation $\kappa\subset X\times X$;
    \item $\a=\neg c(\neg\a)$.
    \end{enumerate}
    \label{complementcor1}
  \item Let $\a\subset X\times X$ be a transitive and reflexive relation such that $\tild\a$ is admissible. The following two statements are equivalent:
    \begin{enumerate}
    \item $\a=\neg \kappa$ for some co-quasiorder $\kappa\subset X\times X$;
    \item $\a=\neg c(\tild\a)$.
    \end{enumerate}
    \label{complementcor2}
  \end{enumerate}
\end{cor}

With respect to \eqref{complementcor2} above, recall from Corollary~\ref{kappanegcompl} that every co-quasiorder $\kappa$ satisfies $\neg\kappa = \tild\kappa$. 

\begin{proof}
The implication (b)$\Rightarrow$(a) is trivial in both cases.

\eqref{complementcor1}
Assume that $\a = \neg\kappa$ for a cotransitive relation $\kappa$. Then $\kappa\subset\neg\a$ and, by Proposition~\ref{cproperties}\eqref{cproperties4}, we have $\kappa\subset c(\neg\a)\subset\neg\a$. Negating this chain of inclusion gives us
\[ \a\subset \neg\neg\a \subset\neg c(\neg\a) \subset \neg\kappa = \a, \]
that is, $\a = \neg c(\neg\a)$.

\eqref{complementcor2}
This is parallel to the proof of \eqref{complementcor1}, using the apartness complement instead of the logical complement.
Assume that $\a = \neg\kappa = \tild\kappa$ for some co-quasiorder $\kappa$. Then $\kappa\subset\tild\a$ and thus $\kappa\subset c(\tild\a)\subset\tild\a$ by Proposition~\ref{cproperties}\eqref{cproperties4}. Taking the apartness complement now gives
\[\a\subset \tild\tild\a \subset\tild c(\tild\a)\subset\tild\kappa = \a \]
and hence $\a = \tild c(\tild\a) = \neg c(\tild\a)$. 
\end{proof}

We can also define a constructive friend of the equivalence closure $\a^e$ of a relation $\a$. 

\begin{dfn}
  Let $\a$ be a relation on $X$ such that $\a\cap\a\inv\cap (\apt)\subset X\times X$ is admissible. The \emph{coequivalence kernel} of $\a$ is the relation $q(\a) = c(\a\cap\a\inv\cap (\apt))$.
\end{dfn}

From Proposition~\ref{cproperties}, one readily deduces the following properties of the coequivalence kernel. 

\begin{prop}\label{qproperties}
Let $\a\subset X\times X$.
  \begin{enumerate}
  \item Let $\mathcal{X}$ be a set of relations on $X$, such that $\b\cap\b\inv\cap(\apt)\subset X\times X$ is admissible for all $\b\in\mathcal{X}$. Then $q$ is a kernel operator on $\mathcal{X}\cup q(\mathcal{X})$.
    \label{qproperties1}
  \item The identity $q(\a) = \a$ holds if and only if $\a$ is a coequivalence.
    \label{qproperties2}
  \end{enumerate}
\end{prop}

It follows that $q(\a)$ is the maximal coequivalence contained in $\a$. In particular, $\a$ is a coequivalence if and only if $q(\a)=\a$.

\begin{rmk}
Some formalisations of constructive mathematics, such as the \emph{Intuitionistic Zermelo--Fraenkel set theory} IZF (see, e.g., \cite[Chapter~VIII]{beeson85}), allow the formation of the power set $\wp(X)$ of an arbitrary set $X$. Within such a framework, one can define the cotransitive kernel $c(\a)$ of \emph{any} relation $\a\subset X\times X$, by
\[
c(\a) = \bigcup \left\{\b\in\wp(\a) \mid \forall_{x,y,z\in X}\,(x\b y \impl x\b z \eller z\b y) \right\} \,.
\]
In other words, $c(\a)$ is the union of all cotransitive subrelations of $\a$. Clearly, defined in this way, $c(\a)$ is the unique maximal cotransitive sub-relation of $\a$. In particular, it coincides with $\tilde{\a}$ if (and only if) $\a$ is admissible.

However, this definition does not provide any kind of algorithm for computing $c(\a)$. For this reason, we find it questionable if it can be given any meaningful constructive interpretation. 
\end{rmk}

%%%%%%%%%%%%%%%%%%%%%%%%%%%%%%%%%%%%%%%%
\subsection{The co-congruence kernel} \label{cocongrkerop}

Let $S$ be a semigroup with apartness.

\begin{dfn} Given a relation $\a$ on $S$, define a new relation $\eta(\a)$ by
\[\eta(\a) = \{(x,y)\in S\times S \mid 
    \forall_{u,v\in S,\:s,t\in S^1}(u\a v \eller (x,y)\!\apt\!(sut,svt)) \} \subset S\times S \,. \]
\end{dfn}

\begin{lma} \label{etalemma}
Let $\a$ and $\b$ be relations on $S$.
  \begin{enumerate}
  \item The relation $\eta(\a)$ is left and right co-compatible with the multiplication
    on $S$; \label{etacocomp}
  \item $\eta(\a)\subset\a$; \label{etashrinks}
  \item if $\a\subset\b$ then $\eta(\a)\subset\eta(\b)$; \label{etainclusion}
  \item if $\a$ is a co-quasiorder that is co-compatible with the multiplication, then $\a = \eta(\a)$. \label{etafix}
  \item An irreflexive and cotransitive relation is co-compatible if and only if it is left co-compatible and right co-compatible. \label{leftandright}
  \end{enumerate}
\end{lma}

\begin{proof}
\eqref{etacocomp}
Given $a,x,y\in S$, we have
\begin{align*}
  (ax)\eta(\a)(ay) \quad\equ\quad
  &\forall_{u,v\in S,\:s,t\in S^1}(u\a v \eller (ax,ay)\apt(sut,svt))  \\
  \quad\stackrel{s=as'}{\impl}\hspace*{-5.5pt}\quad
  &\forall_{u,v\in S,\:s',t\in S^1}(u\a v \eller (ax,ay)\apt(as'ut,as'vt)) \\
  \quad\impl\quad
  &\forall_{u,v\in S,\:s',t\in S^1}(u\a v \eller (x,y)\apt(s'ut,s'vt)) 
  \quad\equ\quad   x\eta(\a)y \,,
\end{align*}
so $\eta(\a)$ is left compatible with the multiplication on $S$. Right compatibility follows similarly.

\eqref{etashrinks}
Assume that $x\eta(\a)y$. Inserting $(u,v)=(x,y)$, $(s,t)=(1,1)$ in the definition of
$\eta(\a)$ gives $x\a y$ or  $(x,y)\apt(x,y)$, hence $x\a y$. 

\eqref{etainclusion}
This is clear from the definition, since $u\a v$ implies $u\b v$.

\eqref{etafix}
Assume that $x\a y$. Then, for all $u,v\in S$ and $s,t\in S^1$, we have
\begin{align*}
  & x\a(sut) \eller (sut)\a(svt) \eller (svt)\a y 
  && \mbox{(by cotransitivity), hence}\\
 & x\apt sut \eller u\a v \eller svt\apt y 
  && \mbox{(by strong irreflexivity and co-compatibility), hence}\\
 & u\a v \eller (x,y)\apt(sut,svt) \,,
\end{align*}
that is, $x\eta(\a)y$ holds.

\eqref{leftandright}
Let $\a$ be an irreflexive and cotransitive relation.
If $\a$ is co-compatible then $(ax)\a(ay)$ implies $a\a a$ or $x\a y$
whence, by irreflexivity, $x\a y$ holds. Similarly, $(ax)\a(bx)$ implies $a\a b$, so $\a$ is left and right compatible.
Conversely, if $\a$ is left and right co-compatible then, by cotransitivity,
\[(ax)\a(by) \:\impl\: (ax)\a(ay) \eller (ay)\a(by)  
\:\impl\: x\a y \eller a\a b   \]
for all $x,y,a,b\in S$. Hence $\a$ is co-compatible. 
\end{proof}

The statements in the lemma below follow readily from the definitions.

\begin{lma} \label{preservescocomp}
  Let $\mathcal{X}$ be a set of relations on $S$. 
 \begin{enumerate}
 \item If every $\a\in\mathcal{X}$ is left (right) co-compatible with the multiplication on $S$, then so are the relations $\bigcap \mathcal{X}$ and
   $\bigcup \mathcal{X}$. 
   \label{capsandcups}
 \item If $\a$ and $\b$ are left (right) co-compatible relations on $S$, then so is $\a*\b$. 
   \label{filled}
 \item If $\a$ is left (right) co-compatible, then $\a\inv$ is right (left) co-compatible.
   \label{inverse}
 \end{enumerate}
\end{lma}

\begin{prop} \label{cocongrprop}
  A relation $\a$ on $S$ is a co-congruence if and only if it is strongly irreflexive and symmetric, and satisfies $\eta(\a)*\eta(\a) = \a$. 
\end{prop}

\begin{proof}
  Let $\a$ be a co-congruence on $S$.
Then $\eta(\a)=\a$ by Lemma~\ref{etalemma}\eqref{etafix} and, since $\a$ is cotransitive, $\eta(\a)*\eta(\a) = \a*\a = \a$ by Lemma~\ref{*properties}\eqref{*properties5}.

For the converse, assume that $\a$ is strongly irreflexive and symmetric, and that $\eta(\a)*\eta(\a)=\a$. By Lemma~\ref{etalemma}(\ref{etacocomp},\ref{leftandright}), $\eta(\a)$ is co-compatible, and from Lemma~\ref{preservescocomp}\eqref{filled}, it follows that so is $\a$. Moreover, by Lemma~\ref{etalemma}\eqref{etashrinks} and Lemma~\ref{*properties}\eqref{*properties6} we have that $\a= \eta(\a)*\eta(\a) \subset \a*\a$, whence Lemma~\ref{*properties}\eqref{*properties4} gives cotransitivity.
\end{proof}

\begin{prop} \label{cqcocomp}
 Let $\a$ be a relation on $S$. 
 \begin{enumerate}
 \item If $\a$ is left (or right) co-compatible, then so is $c(\a)$. \label{ccocomp}
 \item If $\a$ is left and right co-compatible, then so is $q(\a)$.  \label{qcocomp}
\end{enumerate}
\end{prop}

\begin{proof}
Both of these statements follow from Lemma~\ref{preservescocomp}.
Assume that $\a$ is left co-compatible. From Lemma~\ref{preservescocomp}\eqref{filled} it follows by induction every $\b\in\<\a\>^*$ is left co-compatible and thus, by Lemma~\ref{preservescocomp}\eqref{capsandcups}, the same holds for $c(\a) = \bigcap\<\a\>^*$.

For $q(\a)$, Lemma~\ref{preservescocomp} implies that if $\a$ is left and right
co-compatible then $\a\cap\a\inv\cap(\apt)$ is so, too. 
Since these properties are preserved by $c$, it follows that 
$q(\a) = c\left(\a\cap\a\inv\cap(\apt)\right)$ is left and right co-compatible. 
\end{proof}

We can now define a constructive friend of the congruence closure of a relation.

\begin{dfn}
  Let $\a\subset S\times S$ be a relation such that $\eta(\a)\cap\eta(\a)\inv\cap(\apt)$ is admissible. 
   The \emph{co-congruence kernel} of $\a$ is the relation $\zeta(\a) = q\eta(\a)$. 
\end{dfn}

From the properties of $\eta$ and $q$, we deduce the following result.

\begin{prop}\label{zetaproperties}
 Let $\mathcal{X}$ be a set of relations on $S$, such that $\eta(\a)\cap\eta(\a)\inv\cap(\apt)\subset S\times S$ is admissible for all $\a\in\mathcal{X}$.
  \begin{enumerate}
  \item The relation $\zeta(\a)$ is a co-congruence for all $\a\in\mathcal{X}$; \label{zetaproperties1}
  \item $\zeta$ is a kernel operator on $\mathcal{X}\cup \zeta(\mathcal{X})$.
    \label{zetaproperties2}
  \end{enumerate}
\end{prop}

In particular, $\zeta(\a)$ is the maximal co-congruence contained in $\a$.

\begin{proof}
\eqref{zetaproperties1}
Let $\a\in\mathcal{X}$.
By Lemma~\ref{etalemma}\eqref{etacocomp}, $\eta(\a)$ is left and right co-compatible, and hence so is $\zeta(\a)=q(\eta(\a))$, by Proposition~\ref{cqcocomp}. Moreover, $\zeta(\a) = q(\eta(\a))$ is a coequivalence by Proposition~\ref{qproperties}.
In sum, $\zeta(\a)$ is a left and right co-compatible coequivalence, and hence a co-congruence by Lemma~\ref{etalemma}\eqref{leftandright}.

\eqref{zetaproperties2}
Let $\a$ and $\b$ be relations on $S$, such that $\a\subset\b$. 
From Lemma~\ref{etalemma}(\ref{etashrinks},\ref{etainclusion}) follows that $\eta(\a)\subset\a$ and $\eta(\a)\subset\eta(\b)$. 
Since $q$ is a kernel operator by Proposition~\ref{qproperties}\eqref{qproperties1}, it follows that
$\zeta(\a)=q\eta(\a)\subset\eta(\a)\subset\a$, and $\zeta(\a)=q\eta(\a)\subset q\eta(\b)=\zeta(\b)$.
Finally, as $\zeta(\a)$ is a co-congruence by \eqref{zetaproperties1}, we have
\[ \zeta^2(\a) = c(\zeta(\a)\cap\zeta(\a)\inv\cap(\apt)) = c(\zeta(\a)) = \zeta(a) \]
by Proposition~\ref{cproperties}\eqref{cproperties2}.
Hence, $\zeta$ is a kernel operator.
\end{proof}

%%%%%%%%%%%%%%%%%%%%%%%%%%%%%%%%%%%%%%%%
\subsection{Relations and kernels on finitely enumerable sets}  \label{finenumsets}

  In this section, we study relations on finitely enumerable sets and semigroups.
We show that on a such set, every relation $\a$ is cotransitively admissible, and thus the kernels $c(\a)$, $q(\a)$ and $\zeta(\a)$ are always defined.
  We characterise the relations that can be written as negations of co-quasiorders and co-congruences, and derive some facts about the unique maximal apartness relation on a finitely enumerable set.

Throughout this section, unless otherwise stated, let $X=\{x_1,\ldots,x_d\}$ be a finitely enumerable set with apartness.

\begin{thm}\label{finiteisshort}
 Every relation on a finitely enumerable set is short.
\end{thm}

\begin{proof}
Let $\a$ be a relation on $X$.
By Lemma~\ref{fecd}, $X$ has the constant domains property and hence, by Proposition~\ref{cd*assoc}, filled products of relations on $X$ are associative. In particular, it follows that $\<\a\>^* = \{\a^{*l}\mid l\ge1\}$.

We shall prove, by induction on $n$, that $\a\cap\a^{*2}\cap\cdots\cap\a^{*(d+2)} \subset \a^{*(n+1)}$ for all $n\in\N$, whence it follows that $\a\cap\a^{*2}\cap\cdots\cap\a^{*(d+2)} = \bigcap\<\a\>^*$.
The statement is clear for $n\le d+1$.

Let $n\ge d+2$, and $(x,y)\in \a\cap\a^{*2}\cap\cdots\cap \a^{*(d+2)}\subset X\times X$.
We need to show that $x\a^{*(n+1)}y$ which, in view of the constant domains principle \eqref{cdaxiom}, is equivalent to the statement
\[\forall_{z_1,\ldots,z_n\in X} \left(\, (x\a z_1) \eller (z_1\a z_2) \eller \cdots \eller (z_{n-1}\a z_n) \eller (z_n\a y) \, \right) \,.\]

Let $z_1\ldots, z_n\in X$. Since $n-1> d$, there exist $i,j\in\{1,\ldots,n-1\}$ such that $i<j$ and $z_{i+1} = z_{j+1}$.
Now, as $n-j+i+1 < n+1$, we have $\a\cap\a^{*2}\cap\cdots\cap \a^{*(d+2)}\subset \a^{*(n-j+i+1)}$ by the induction hypothesis, and thus $(x,y)\in\a^{*(n-j+i+1)}$. This means that, for all $w_1,\ldots, w_{n-j+i}$, the disjunction
\[ (x\a w_1) \eller (w_1\a w_2) \eller \cdots \eller (w_{n-j+i-1}\a w_{n-j+i}) \eller (w_{n-j+i}\a y) \,. \]
holds. In particular, setting
\[ w_l =
\begin{cases}
  z_l & \mbox{for } l\le i, \\ z_{l+j-i} & \mbox{for } l>i,
\end{cases}
\]
and recalling that $z_{i+1} = z_{j+1}$, we get
\[ (x\a z_1) \eller (z_1\a z_2) \eller \cdots \eller (z_i\a z_{i+1}) \eller (z_{j+1}\a z_{j+2})\eller \cdots  \eller (z_{n-1}\a z_n) \eller (z_n\a y)\]
and hence
\[ (x\a z_1) \eller (z_1\a z_2) \eller \cdots \eller (z_{n-1}\a z_n) \eller (z_n\a y) \,. \]
This proves that $x\a^{* (n+1)}y$. Consequently, the inclusion $\a\cap\a^2\cap \cdots\cap \a^{*(d+2)}\subset\a^{*(n+1)}$ holds for all $n\in\N$, concluding the proof of our result.
\end{proof}

Combining Proposition~\ref{atilde-cotrans}\eqref{atilde-cotrans1} with Theorem~\ref{finiteisshort} immediately gives the following result.

\begin{cor}
  Every relation on a finitely enumerable set is admissible.
\end{cor}

We can now give a refinement of Corollary~\ref{complementcor}, characterising the relations arise as negations of cotransitive relations and co-quasiorders. 
Note that if a relation $\a$ can be written as $\a=\neg\b$ for \emph{any} relation $\b$, then
\[ \neg\neg\a = \neg\neg\neg\b = \neg\b = \a \]
that is, $\a$ is stable. Conversely, if $\a$ is stable then $\a = \neg(\neg\a)$. It now turns out that for a reflexive and transitive relation  $\a$, stability is equivalent to being the complement of a \emph{cotransitive} relation.

\begin{lma} \label{negneglma}
  Let $\a\subset X\times X$ be an irreflexive and transitive relation on $X$. Then
  \begin{enumerate}
  \item $\neg\b = \neg\neg\a$ for all $\b\in\<\neg\a\>^*$; \label{negneglma1}
  \item $\neg\b = \neg\neg\a$ for all $\b\in\<\tild\a\>^*$, if the apartness $\apt$ on $X$ is standard. \label{negneglma2}
  \end{enumerate}
\end{lma}

\begin{proof}
In both cases, $\b\subset\neg\a$ and thus $\neg\neg\a\subset\neg\b$. It remains to prove that $\neg\b\subset\neg\neg\a$ or, equivalently, $(\neg\a)\cap(\neg\b)=\emptyset$.

\eqref{negneglma1}
Assume, for a contradiction, that $(\neg\a)\cap(\neg\b)$ is inhabited.
Assume further that $\a\subset X\times X$ is detachable, that is, that
\begin{equation}
  \bigwedge_{i,j=1}^d\left(\, x_i\a x_j \eller \neg(x_i\a x_j)\,\right) \,. \label{glivenko}
\end{equation}
In this case, the transitivity of $\a$ implies that $\neg\a$ is cotransitive and thus $\neg\a\subset(\neg\a)*(\neg\a)$ by Lemma~\ref{*properties}\eqref{*properties4}. By induction, we conclude that $\neg\a\subset\b$. But then $\neg\b\subset\neg\neg\a$, and hence $(\neg\a)\cap(\neg\b) \subset (\neg\neg\a)\cap(\neg\a) = \emptyset$, which is impossible since $(\neg\a)\cap(\neg\b)$ is inhabited.
This proves $\neg \bigwedge_{i,j=1}^d\left(\, (x_i\a x_j) \eller \neg(x_i\a x_j)\,\right)$, which again is impossible, by Glivenko's theorem \cite{glivenko29}.
Thus $(\neg\a)\cap(\neg\b)$ is not inhabited, that is, $(\neg\a)\cap(\neg\b) = \emptyset$.

\eqref{negneglma2}
This is a variation of the proof of \eqref{negneglma1}. Assume that  $(\neg\a)\cap(\neg\b)$ is inhabited and that the statement
\begin{equation} \label{glivenko2}
  \bigwedge_{i,j=1}^d\left(\, x_i\a x_j \eller x_i(\tild\a) x_j\,\right)
\end{equation}
holds. Then $\tild\a = \neg\a$, and this relation is cotransitive. As before, we get that $\neg\a\subset\b$, whence $\neg\b\subset\neg\neg\a$ and $(\neg\a)\cap(\neg\b)\subset(\neg\a)\cap(\neg\neg\a)=\emptyset$, which is impossible since $(\neg\a)\cap(\neg\b)$ is inhabited. Thus
$\neg \bigwedge_{i,j=1}^d\left(\, x_i\a x_j \eller x_i(\tild\a) x_j\,\right)$
holds.
But since the apartness is standard, the proposition \eqref{glivenko2} is classically tautological, and therefore
$\neg \bigwedge_{i,j=1}^d\left(\, x_i\a x_j \eller x_i(\tild\a) x_j\,\right)$
is impossible by Glivenko's theorem. So $(\neg\a)\cap(\neg\b)=\emptyset$, as desired.
\end{proof}

\begin{thm} \label{complementthm}
  Let $\a$ be a reflexive and transitive relation on $X$.
  The following statements are equivalent:
  \begin{enumerate} \renewcommand{\labelenumi}{(\alph{enumi})}
  \item $\a = \neg\neg\a$;
  \item $\a=\neg\b$ for some cotransitive relation $\b\subset X\times X$.
  \end{enumerate}
If the apartness on $X$ is standard, then the above is also equivalent to the following statement:  
  \begin{enumerate}  \renewcommand{\labelenumi}{(\alph{enumi})} \addtocounter{enumi}{2}
  \item $\a=\neg\kappa$ for some co-quasiorder $\kappa\subset X\times X$.
  \end{enumerate}
\end{thm}

\begin{proof}
The implications (c)$\Rightarrow$(b)$\Rightarrow$(a) are trivial.
Since $\a$ is reflexive, $\neg\a$ is irreflexive, and thus $c(\neg\a)\in\<\a\>^*$. By Lemma~\ref{negneglma}\eqref{negneglma1}, this implies that $\neg c(\neg\a) = \neg\neg\a$. Since $c(\neg\a)$ is cotransitive, this gives the implication (a)$\Rightarrow$(b).
Similarly, if the apartness is standard then $\neg c(\tild\a) = \neg\neg\a$ by Lemma~\ref{negneglma}\eqref{negneglma2} and, since $c(\tild\a)$ is a co-quasiorder, the implication (a)$\Rightarrow$(c) follows.
\end{proof}

\begin{rmk}
The implication (a)$\Rightarrow$(c) does not hold in any case unless the apartness on $X$ is standard. Indeed, let $\a=\neg\neg(=)$, and assume that $\a=\neg\kappa$ for some strongly irreflexive relation $\kappa$ on $X$. Then $\kappa\subset(\apt)$ implies $\neg(\apt)\subset\neg\kappa = \a = \neg\neg(=)$, and thus $\neg(\apt)\cap \neg(=) = \emptyset$, that is, $\apt$ is standard.
\end{rmk}

The \emph{fine apartness} on $X$ is the relation $(\bowtie) = c(\neg(=))\subset X\times X$.  This the maximal irreflexive, symmetric and cotransitive relation on $X$, and hence any other apartness relation on $X$ is contained in $\bowtie$. 
If $X=(X,\apt)$ is a set with apartness, we shall say that the apartness is fine if
$(\apt) = (\bowtie)$. 

\begin{prop} \label{fineapartness}
  If $X$ is a set with fine apartness, then every detachable subset of $X$ is strongly extensional, and thus clopen in the apartness topology.
\end{prop}

A subset $A\subset X$ is clopen in the apartness topology if both $A$ and $\tild A$ are strongly extensional in $X$ (see Section~\ref{fundamental}). 

\begin{proof}
Assume that $\apt$ is fine. Given a detachable subset $A\subset X$, define
a relation $\kappa$ on $X$ by 
\[x\kappa y \:\Leftrightarrow\: x\apt y \eller (x\in A\not\ni y) \eller (x\notin A\ni y) \,.\]
One readily verifies that $\kappa$ is an apartness relation with respect to which
$A\subset X$ is strongly extensional. But since $(\apt)\subset \kappa$ and $\apt$ is fine, we have $(\apt)=\kappa$. Hence $A\subset X$ is strongly extensional with respect to $\apt$.

As $\tild A=\neg A\subset X$ is also detachable and thus strongly extensional, it follows
that $A\subset X$ is clopen. 
\end{proof}

The following lemma is straightforward.

\begin{lma} \label{weakapartness}
Given any sets $Y$ and $Z$, a function $f:Y\to Z$ and an apartness relation $\apt$ on $Z$, define a relation $\apt_f$ on $Y$ by $x\apt_f y \:\equ f(x)\apt f(y)$. Then $\apt_f$ is an apartness relation on $Y$, and $f:(Y,\apt_f)\to (Z,\apt)$ is strongly extensional.
\end{lma}

\begin{prop} \label{allse}
\begin{enumerate}
\item Let $X,Y\in\aSets$, where $X$ is finitely enumerable and has fine apartness. Then any function $f:X\to Y$ is strongly extensional.
  \label{allse1}
\item If $S$ is a finitely enumerable semigroup, then $(S,\bowtie)$ is a semigroup with apartness.
  \label{allse2}
\item Let $S$ be a semigroup with apartness, and $\rho\subset S\times S$ a congruence such that $S/\rho$ is finitely enumerable. Then $c(\neg\rho)$ is a co-congruence on $S$, and hence $c(\neg\rho)= \zeta(\neg\rho)$. 
  \label{allse3}
\end{enumerate}
\end{prop}

Note that if $S\in\aSg$ is finitely enumerable, then Proposition~\ref{allse}\eqref{allse3} applies to \emph{every} congruence $\rho$ on $S$.

\begin{proof}
\eqref{allse1}
This follows from Lemma~\ref{weakapartness}: if the apartness on $X$ is fine, then it contains the apartness $\apt_f$, and thus $f$ is strongly extensional. 

\eqref{allse2}
By \eqref{allse1}, every function $(S,\bowtie)\to (S,\bowtie)$ is strongly extensional. In particular, this applies to the functions $\lambda_a:x\mapsto ax $ and $\rho_a:x\mapsto xa$ for any $a\in S$, whence it follows that the multiplication in $S$ is strongly extensional.

\eqref{allse3}
The factor set $S/\rho$ is finitely enumerable, with fine apartness given by  $(\bowtie)=c(\neg\rho)$. 
By \eqref{allse2}, the multiplication in $S/\rho$ is strongly extensional with respect to $\bowtie$, which is to say that $c(\neg\rho)$ is co-compatible with the multiplication. Thus $c(\neg\rho)$ is a co-congruence.
Since $\eta(\neg\rho)\subset\neg\rho$ (by Lemma~\ref{etalemma}\eqref{etashrinks}) and $q$ is a kernel operator (Proposition~\ref{qproperties}), we have $\zeta(\neg\rho) = q(\eta(\neg\rho))\subset q(\neg\rho) = c(\neg\rho)$.
On the other hand, $c(\neg\rho)$ is a co-congruence contained in $\neg\rho$, so $c(\neg\rho)\subset\zeta(\neg\rho)$ by Proposition~\ref{zetaproperties}.
Hence $c(\neg\rho)=\zeta(\neg\rho)$. 
\end{proof}

We spell out the following two consequences of Proposition~\ref{allse}.

\begin{cor}
  A congruence on a finitely enumerable semigroup with apartness is the complement of a co-congruence if and only if it is stable. 
\end{cor}

\begin{proof}
  Immediate from Proposition~\ref{allse}\eqref{allse3} and Theorem~\ref{complementthm}.
\end{proof}

Let $\Sets_{\rm fe}$, $\Sets^{\apt}_{\rm fe}$, $\Sg_{\rm fe}$ and $\Sg^{\apt}_{\rm fe}$ be the full subcategories formed by all finitely enumerable objects in $\Sets$, $\aSets$, $\Sg$ and $\aSg$, respectively. 
For all $X,Y\in\Sets_{\rm fe}$ and $f:X\to Y$, let $A(X) = (X,\bowtie)$ and $A(f)=f$.
Denote by $V:\Sets^{\apt}_{\rm fe}\to\Sets_{\rm fe}$, and $V:\Sg^{\apt}_{\rm fe}\to\Sg_{\rm fe}$ the forgetful functors. 

\begin{cor}
  The assignment $A$ defines a full and faithful functor $A:\Sets\to\aSets$, satisfying $V\circ A = \I_{\Sets}$. It induces a full and faithful functor $A:\Sg\to\aSg$ satisfying $V\circ A = \I_{\Sg}$.
\end{cor}

\begin{rmk} \label{finermk}
\begin{enumerate}
\item
  By Theorem~\ref{complementthm}, the fine apartness it tight if and only if the equality relation is stable. Similarly, if the fine apartness relation is stable, then
  \[(\bowtie) = c(\neg(=)) = \neg\neg c(\neg(=)) = \neg\neg\neg(=) = \neg(=) \] by Lemma~\ref{negneglma}. So $\bowtie$ is stable if and only if $\neg(=)$ is cotransitive.
  \label{finermk1}
\item
  Neither of the two properties \emph{tight} respectively \emph{stable} implies the other for the fine apartness relation $\bowtie$. For example, let $X=\{0,1\}$ with equality relation $\e$ defined by $0\e 1 \:\equ\: P\eller\neg P$, where $P$ is some proposition. Then $0(\neg\e)1 \:\equ\: \neg(P\eller\neg P) \:\equ\: \bot$, and thus $ \neg\e = \emptyset$, which is a cotransitive relation. So $\bowtie$ is stable, by \eqref{finermk1}. On the other hand, $\bowtie$ is tight if and only if $\e$ is stable, if and only if $P$ is decidable.

  Next, let $Y=\{0,1,2\}$ with equality relation $\d$ determined by $0\d 1 \:\equ\: \neg P$, $1\d 2 \:\equ\: \neg\neg P$, and $\neg(0\d2)$.
  Clearly, the relation $\d$ is stable, so $\bowtie$ is tight. On the other hand, if $\neg\d$ is cotransitive then either $0(\neg\d)1$ or $2(\neg\d)2$ holds, that is, either $\neg\neg P$ or $\neg P$. Thus the implication ``\,$\bowtie$ is tight'' $\:\Rightarrow\:$ ``\,$\bowtie$ is stable'' entails WLEM. 
  \label{finermk2}
\item 
  The conclusion of Proposition~\ref{fineapartness} does not hold if we replace ``fine'' with ``tight'' in its premise. For let $P$ be any proposition, and define an apartness $\apt$ on $X=\{0,1\}$ by $0\apt1\:\equ\: (P\eller\neg P)$. 
  Now $\neg(0\apt1)$ is equivalent to $\neg (P\eller \neg P)$ which is impossible; consequently, $\neg(x\apt y)$ holds only if $x=y$. The subset $\{0\}\subset X$ is clearly detachable, but not strongly extensional unless $P\eller\neg P$ holds. Thus, the statement that every detachable subset of $X$ is strongly extensional entails LEM.
  \label{finermk3}
\item 
  The apartness relation $\apt$ in \eqref{finermk3} is an example of a tight apartness that is not necessarily fine.
    \label{finermk4}
  \item
    By Proposition~\ref{fineapartness}, every detachable subset of $X$ is clopen with respect to the topology of the fine apartness relation on $X$. However, it is not possible to prove in general that clopen sets are detachable, or even stable. For example, given a proposition $P$, let $X=\{0\}$ and $A_P=\{x\in X\mid P \}\subset X$. Then $A_P$ and $\neg A_P=\{x\in X \mid \neg P\}$ are strongly extensional in $X$ (with respect to the unique apartness relation $(\apt) = \emptyset$ on $X$), and $\neg\neg A_P=\{x\in X \mid \neg\neg P\}$. So $A_P\subset X$ is detachable if and only if $P$ is decidable, and stable if and only if $P$ is stable.
  \label{finermk5}
\end{enumerate}
\end{rmk}

%%%%%%%%%%%%%%%%%%%%%%%%%%%%%%%%%%%%%%%%
\subsection{Not every relation is cotransitively admissible}
\label{notadmissible}

The purpose of this section is to prove the following result. 

\begin{thm}\label{irreflexivelpo}
  The statement that every irreflexive relation is cotransitively admissible is equivalent to the limited principle of omniscience (LPO).
\end{thm}

As usual, we write $\tilde{\a} = \bigcap\<\a\>^{*}$.

\begin{proof}[Proof of the implication ``$\Leftarrow$'']
Given an irreflexive relation $\a$ on a set $X$, define relations $\{\a_n\}_{n\in\N}$ inductively by $\a_0=\a$ and $\a_{n+1} = \a_n*\a_n$.
From Lemma~\ref{*properties}(\ref{*properties3},\ref{*properties6}), it follows that  $\a_n$ is irreflexive and $\a_{n+1}\subset\a_n$ for all $n\in\N$.
Let $\b\subset\a_l$ and $\gamma\subset\a_m$ for some $l,m\in\N$. Then $\b,\gamma\subset\a_n$ for $n=\max\{l,m\}$, and thus $\b*\gamma\subset\a_n*\a_n = \a_{n+1}$ by Lemma~\ref{*properties}\eqref{*properties6}. By induction, it follows that every $\b\in\<\a\>^*$ is contained in $\a_n$ for some $n\in\N$, and hence that $\tilde{\a} = \bigcap\<\a\>^* = \bigcap\{\a_n\mid n\in\N\}$.

Let $x,y,z\in X$, and assume that $x\tilde{\a}y$.
Set $A=\{n\in\N \mid x\a_n z\}$ and  $B=\{n\in\N \mid z\a_n y\}$.
For each $n\in\N$, we have $(x,y)\in\a_{n+1}=\a_n*\a_n$ and hence $x\a_nz$ or $z\a_ny$. This means that $A\cup B =\N$. By \cite[Proposition~1.2.4]{diener20}, LPO now implies that either $A$ or $B$ is infinite. Since $n\in A \:\impl\: \forall_{m\le n}\,(m\in A)$, infiniteness of $A$ is equivalent to $A=\N$. Similarly,  $B$ is infinite if and only if $B=\N$. This proves that $A=\N$ or $B=\N$, that is, $(x,z)\in\bigcap\{\a_n\mid n\in\N\}=\tilde{\a}$ or $(z,y)\in\bigcap\{\a_n\mid n\in\N\}=\tilde{\a}$. Hence $\tilde{\a}$ is cotransitive.
\end{proof}

\begin{rmk}
Proposition~1.2.4 in \cite{diener20}, which is used in the above proof,  relies on the axiom of countable choice.
\end{rmk}

Define a relation $\a$ on  $X = \{0,1\}^{\mathbb{N}}$ by
\[x\a y \:\equ\: \exists_{i,j\in\N} ( \, i\ne j \och x_i\ne y_i \och x_j\ne y_j \,)\,.\]
To conclude  the proof of Theorem~\ref{irreflexivelpo}, we shall prove that cotransitivity of $\tilde{\a}$ implies LPO.

\begin{prop} \label{verydifferent}
  Let $x,y\in X$. Then $x\tilde{\a} y$ if and only if there exists an injective function $f:\N\to N$ such that $x_n\ne y_n \:\equ\: \exists_{m\in\N}: n=f(m)$ holds for all $n\in\N$.
\end{prop}

We shall use the notation $\N_n = \{0,\ldots, n\}$ and $X_n = \{0,1\}^{\N_n}$. Note that the sets $\N_n$, $X_n$ and $\{t\in\N_n \mid x_t\ne y_t\}$ are finite for all $n\in\N$. 

\begin{proof}[Proof of ``$\Rightarrow$'']
Let $n\in\N$. Set $\{t_1,\ldots, t_r\} = \{t\in\N_n \mid x_t\ne y_t\}$.
Define $z^{(0)},\ldots,z^{(r)}\in X$ by $z^{(0)} = x$ and 
\[
z_i^{(s)} =
\begin{cases}
  z_i^{(s-1)} & \mbox{if } i\ne t_s \,,\\
  y_i & \mbox{if } i = t_s \,.
\end{cases}
\]
for $1\le s\le r$. 
Then $\neg(z^{(s)}\a z^{(s+1)})$ for all $s\in\N_{r-1}$, that is,
$\neg(x\a z^{(0)} \eller z^{(0)}\a z^{(1)} \eller \cdots z^{(r-1)}\a z^{(r)} )$.
On the other hand, $x\tilde{\a}y$ implies
$x\a z^{(0)} \eller z^{(0)}\a z^{(1)} \eller \cdots \eller z^{(r-1)}\a z^{(r)} \eller z^{(r)}\a y$; hence, $z^{(r)}\a y$. In particular, there exists some $m\in\N$ such that $z^{(r)}_m \ne y_m$. 
By construction, $z^{(r)}_i = y_i$ for $i\le n$, and  $z^{(r)}_i = x_i$ for $i> n$. Thus, $m>n$ and $x_m = z^{(r)}_m\ne y_m$.

We have proved that for every $n\in\N$ there exists a number $m>n$ such that $x_m\ne y_m$. Now define a function $f:\N\to\N$ recursively by
\begin{align*}
  f(0) &= \min\{i\in\N_l \mid x_i\ne y_i \}, &&\mbox{where $l\in\N$ is such that $x_l\ne y_l$}; \\
    f(n+1) &= \min\{ i\in\N_m\setminus\N_{f(n)} \mid x_i\ne y_i \}, &&\mbox{where $m>f(n)$ is such that $x_m\ne y_m$}.
  \end{align*}
Injectivity of $f$ is immediate from the definition.
The property $x_m\ne y_m \:\equ\: m\in \im f$ follows, by induction on $m$, from the fact that the set $\{t\in\N_m \mid x_t\ne y_t\}$ is finite.
\end{proof}

For the proof of the ``if'' part of Proposition~\ref{verydifferent}, we shall use some auxiliary notation and results.
For a positive integer $d$ and a relation $\b$ on a set $Y$, we denote by $\<\b\>^{*d}$ the degree $d$ part of $\<\b\>^*$, that is,
\[
\<\b\>^{*1} = \{\b\}, \qquad\mbox{and}\qquad \<\b\>^{*d} = \bigcup_{l+m=d}\left\{\b_1*\b_2 \mid \b_1\in\<\b\>^{*l},\: \b_2\in\<\b\>^{*m}\right\} \quad\mbox{for}\; d>1.
\]
Clearly, $\<\b\>^* = \bigcup_{d\ge1} \<\b\>^{*d}$ and thus, given $a,b\in Y$, the relation $a\tilde{\b}b$ holds if and only if $(a,b)\in\bigcap\<\b\>^{*d}$ for all $d\ge 1$.
Observe that, in the case $Y=X_n$, the filled product is associative by Lemma~\ref{fecd} and Proposition~\ref{cd*assoc}, and thus $\<\b\>^{*d} = \{\b^{*d}\}$ for all $\b\subset X_n\times X_n$ and $d\ge1$. 

Given $n\in\N$, define $\tau_n: X\to X_n,\: a\mapsto a'=(a_i)_{i\le n}$. For $u,v\in X_n$, let
\[\d(u,v) = |\{i\in\N_n\mid u_i\ne v_i\}|\in\N\]
be the cardinality of the finite set $\{i\in\N_n\mid u_i\ne v_i\}$.
A relation $\a_n$ on $X_n$ is defined by $u\a_n v \:\equ\: \d(u,v)\ge 2$.

\begin{lma} \label{liftalpha}
  Let $a,b\in X$, and $d\ge 1$.  If $(\tau_n(a),\tau_n(b))\in \a_n^{*d}$ for some $n\in\N$, then $(a,b)\in\bigcap\<\a\>^{*d}$.
\end{lma}

\begin{proof}
Let $\b \in\<\a\>^{*d}$. Then $\b = g(\a)$ for some non-associative and non-commutative polynomial $g(T)$ of degree $d$. Consequently, $g(\a_n)\in\< \a_n\>^{*d} =\{\a_n^{*d}\}$, i.e., $g(\a_n) = \a_n^{*d}$.
  
Assume that $d=1$ and thus $g(T) = T$. Then $(\tau_n(a),\tau_n(b))\in \a^{*d} \a_n$ means that there exist $i\ne j \in \N_n$ such that $a_i\ne b_i$ and $a_j\ne b_j$. Hence, $(a,b) \in \a = g(\a)$.

Assume that $d>1$ and $g(T) = g_1(T)g_2(T)$. Let $z\in X$. Since
\[
(a',b') = (\tau_n(a),\tau_n(b))\in \a_n^{*d}= g(\a_n) = g_1(\a_n)*g_2(\a_n)\,,
\]
it follows that either $(a',z')\in g_1(\a_n)$ or $(z',b')\in g_2(\a_n)$ holds. By induction we infer that $(a,z)\in g_1(\a)$ in the first case, and $(z,b)\in g_2(\a)$ in the second. Hence $(a,b)\in g_1(\a)*g_2(\a) = g(\a)=\b$.

We conclude that $(a,b)\in\b$ for all $\b\in\<\a\>^{*d}$, that is, $(a,b)\in\bigcap\<\a\>^{*d}$.
\end{proof}

\begin{proof}[Proof of the implication ``$\Leftarrow$'' in Proposition~\ref{verydifferent}]
Let $n\in\N$, and set $l=f(n+2)$. For all $z_1,\ldots,z_n\in X_{l}$, we have
  \[ \d(x',z_1) + \d(z_1,z_2) +\cdots+ \d(z_{n-1},z_n) + \d(z_n,y') \ge \d(x',y') = n+2 \,.\]
Using the finiteness of $\N_{l}$, it follows that at least one term in this sum has to be greater than or equal to $2$, which means that the disjunction
\[ x'\a_{l} z_1 \eller z_1\a_{l}z_2 \eller \cdots\eller z_{n-1}\a_{l}z_n \eller z_n\a_{l}y' \]
holds. This proves that $(x',y')\in \a_{l}^{*n}$, and hence $(x,y)\in\bigcap\<\a\>^{*n}$, by Lemma~\ref{liftalpha}. Since $n$ was arbitrary, it follows that $(x,y)\in\bigcap\<\a\>^* = \tilde{\a}$. 
\end{proof}

The following result proves the ``only if'' part of Theorem~\ref{irreflexivelpo}.

\begin{prop}
  If $\tilde{\a}$ is cotransitive then LPO holds.
\end{prop}

\begin{proof}
From Proposition~\ref{verydifferent}, it is clear that $0\tilde{\a}1$, and hence $0(\tilde{\a}*\tilde{\a})1$ by Lemma~\ref{*properties}\eqref{*properties4}. 

Given a sequence $a\in X$, define $b\in X$ by $b_n = \max\{a_i\mid i\in \N_n\}$ for all $n\in\N$.
Then
\begin{align*}
  0\tilde{\a}b &\:\equ\: \exists_{n\in\N}\,b_n=1 \:\equ\: \exists_{n\in\N}\,a_n=1,\quad \mbox{and}\\
  b\tilde{\a}1 &\:\equ\: \neg\exists_{n\in\N}\,b_n=1 \:\equ\: \forall_{n\in\N}\,a_n=0.  
\end{align*}
Now $0(\tilde{\a}*\tilde{\a})1$ implies that either $0\tilde{\a}b$ or $b\tilde{\a}1$, that is, either $\exists_{n\in\N}\,a_n=1$ or $\forall_{n\in\N}\,a_n=0$.
 Thus LPO holds.
\end{proof}

%%%%%%%%%%%%%%%%%%%%%%%%%%%%%%%%%%%%%%%%%%%%%%%%%%%%%%%%%%%%%%%%%%%%%%%%%%%%%%%%
\section{Some constructions on semigroups with apartness} \label{semigroups}

In this section, we consider some basic constructions on semigroups with apartness: free semigroups, monogenic and periodic semigroups, and Rees factors, and prove some results about the set of idempotents.

%%%%%%%%%%%%%%%%%%%%%%%%%%%%%%%%%%%%%%%%
\subsection{Free semigroups with apartness}

Given a set with apartness $X$, let $X^*$ be the set of finite sequences of elements
(words) in $X$, with the natural equality relation. 
For two elements 
$\bar{x}=x_1x_2\cdots x_m$ and $\bar{y}=y_1y_2\cdots y_n$ in $X^*$, set
$\bar{x}\apt \bar{y}$ if either $m\ne n$ or $(m=n)\wedge\exists_{i\le m}(x_i\apt y_i)$. 
It is straightforward to see that this defines an apartness relation on $X^*$, which is tight if and only if the apartness on $X$ is tight.

Let $F(X)=X^*$ be the semigroup with multiplication given by juxtaposition of words:
$\bar{x}\cdot\bar{y}=\bar{x}\bar{y}=x_1\cdots x_my_1\cdots y_n$. 
Clearly, this operation is strongly extensional, so $F(X)$ is a semigroup with
apartness. 

Given a strongly extensional map $f:X\to Y$ between sets with apartness, define
$F(f):X^*\to Y^*$ inductively by
\[
F(f)(x)=f(x) \mbox{ for $x\in X$}, \qquad 
F(f)(\bar{x}x) = F(f)(\bar{x})\cdot f(x) \mbox{ for $\bar{x}\in X^*$, $x\in X$}.
\]
Again, it is straightforward to verify that $F(f):F(X)\to F(Y)$ is a strongly extensional
morphism, $F(1_X) = 1_{F(X)}$ and that $F(gf)=F(g)F(f)$ for all strongly extensional maps 
$f:X\to Y$, $g:Y\to Z$.
Thus, $F$ is a functor from $\aSets$ to $\aSg$.

\begin{thm} \label{freethm}
  The functor $F:\aSets\to \aSg$ is left adjoint to the forgetful functor $V:\aSg\to\aSets$. 
\end{thm}

Denote by $\iota_X:X\to F(X),\:x\mapsto x$ the natural inclusion of $X$ into the set $F(X)=X^*$.
The usual extension property for maps out of a free object follows immediately from Theorem~\ref{freethm}:

\begin{cor} \label{freecor}
Let $X$ be a set with apartness, and $S$ a semigroup with apartness. 
For every strongly extensional function $f:X\to S$ there exists a unique strongly extensional morphism $\phi:F(X)\to S$ such that $f=\phi\iota_X$.
\end{cor}

\begin{proof}[Proof of Theorem~\ref{freethm}]
The proof is the usual one; we just need to verify that everything makes sense
constructively, and that apartness is preserved.

Given $X\in\aSets$, $S\in\aSg$ and $f\in\aSets(X,V(S))$, define
$\eta_{(X,S)}(f)=\phi:F(X)\to S$ by  $\phi(x)=f(x)$ and
$\phi(\bar{x}x)=\phi(\bar{x})f(x)$ for $x\in X$, $\bar{x}\in X^*$.
By associativity, it follows that $\phi(\bar{x}\bar{y})=\phi(\bar{x})\phi(\bar{y})$ holds for all
$\bar{x},\bar{y}\in F(X)$. To see that $\phi$ is strongly extensional, assume that
$\phi(\bar{x})\apt \phi(\bar{y})$ for some $\bar{x}=x_1\cdots x_m,\bar{y}=y_1\cdots y_n\in
F(X)$.
If $m\ne n$ then $\bar{x}\apt \bar{y}$ and we are done. 
If $m=n=1$ then $f(x_1)=\phi(\bar{x})\apt \phi(\bar{y})=f(y_1)$, and thus 
$\bar{x}=x_1\apt  y_1=\bar{x}$ by strong extensionality of $f$. 
Assume that $m=n>1$, so that $\bar{x}=\bar{x}'x$ and $\bar{y}=\bar{y}'y$ for some
$\bar{x}', \bar{y}'\in X^*$, $x,y\in X$. Then 
\[
\phi(\bar{x}')\phi(x) = \phi(\bar{x})\apt \phi(\bar{y})=\phi(\bar{y}')\phi(y)
\] 
and thus either $\phi(\bar{x}')\apt\phi(\bar{y}')$ or $\phi(x)\apt\phi(y)$. 
By induction we conclude that either $\bar{x}'\apt\bar{y}'$ or $x\apt y$; in
both cases, $\bar{x}\apt\bar{y}$.
This shows that $\eta_{(X,S)}(f)\in\aSg(F(X),S)$.

As in the classical situation, one proves that $\eta:\aSets(-,V(?)) \to \aSg(F(-),?)$ is a
natural transformation.

For an inverse of $\eta$, define $\zeta_{(X,S)}:\aSg(F(X),S) \to \aSets(X,V(S))$ by 
$\zeta_{(X,S)}(\phi) = \phi\iota_X:X\to$ for all $X\in\aSets$, $S\in\aSg$ and
$\phi\in\aSg(F(X),S)$. As $\iota_X$ and $\phi$ are strongly extensional, so is 
$\zeta_{(X,S)}(\phi)$. 
Again, one verifies that the above defines a natural transformation 
$\zeta:\aSg(F(-),?) \to \aSets(-,V(?))$, which is inverse to $\eta$. 

Hence $\aSg(F(-),?) \simeq \aSets(-,V(?))$, that is, $(F,V)$ is an adjoint pair. 
\end{proof}

%%%%%%%%%%%%%%%%%%%%%%%%%%%%%%%%%%%%%%%%
\subsection{Periodicity and monogenic semigroups}

A semigroup $S$ with apartness is \emph{monogenic} if it is generated by a single element,
that is, if there exists a surjective morphism $\f$ from the free semigroup 
$F(\{a\})\simeq (\Z_{>0},+)$ to $S$.
By Proposition~\ref{basicfactor}, such a semigroup $S$ is isomorphic to $F(\{a\})/\ker\f$,
with apartness given by the relation $\cker\f$. Hence, every monogenic semigroup is
specified by a pair $(\rho, \zeta)$ consisting of a congruence and a co-congruence on
$(\Z_{>0},+)$, satisfying $\rho\cap\zeta=\emptyset$. 

In the classical case, every congruence $\rho$ on the additive semigroup $\Z_{>0}$ has the
form 
$\rho = \rho_{r,d}=\{(x,y)\in \Z_{>0}\times\Z_{>0} \mid x,y\ge r
\,\wedge\, (x-y)\in d\Z\}$
for some positive integers $r$ and $d$. This description relies on the result that
every non-empty subset of $\Z_{>0}$ has a smallest element, and hence is not valid
constructively. 
However, some properties still carry over from the classical situation.

Let $S$ be a semigroup, not necessarily monogenic.
An element $a$ in $S$ is said to be \emph{periodic} if there exist positive integers
$r$ and $m$ such that $a^{m+r}=a^m$. In this case, $r$ is a \emph{period} of $a$.
The semigroup $S$ is said to be periodic if all its elements are periodic. 
Obviously, set set of periods of $a$ is closed under addition, and under multiplication
with arbitrary elements in $\Z_{>0}$. 
Moreover, the following properties are easily verified. 

\begin{prop} \label{periodicprop}
Let $a\in S$, and $m,r,s\in\Z_{>0}$.
\begin{enumerate}
\item If $a^{m+r} = a^m$ then $a^{n+r} = a^n$ for all $n\ge m$.
\item If $r$ and $s$ are periods of $a$, then so is $\gcd(r,s)$.
\item A monogenic semigroup $\<a\>$ is periodic if and only if it contains a periodic element. In this case, all elements in $\<a\>$ have the same periods.
\item If $a^{m+r}=a^m$ then $\{a^n \mid n\ge m\}$ is a cyclic subgroup of $\<a\>$.
\item If $a$ is periodic then there exists a positive integer $n$ such that $a^n$ is an idempotent. \label{periodicprop5}
\item Every subfinite semigroup, and every finitely enumerable semigroup, is periodic. 
\end{enumerate}
\end{prop}

To illustrate the additional considerations that the constructive frameworks leads us to, consider the following example. 

\begin{ex}
 Let $P$ be a proposition, and $\rho_P$ the congruence relation on $\Z_{>0}=(\Z_{>0},+)$ defined by  
 \[m\rho_P n \;\Leftrightarrow\; 
 (P \,\wedge\, (m-n)\in 2\Z) \,\eller\, (\neg P \,\wedge\,(m-n)\in3\Z), \] 
 and $S_P=\Z_{>0}/\rho_P$. 
 For any two positive integers $m$ and $n$, the statement $m\rho_Pn$ entails $P\eller\neg P$ so, in order to establish the periodicity of an element in $S_P$, we must prove either $P$ or its negation. 
 
Moreover, the existence of an inhabited apartness relation $\apt$ on $S_P$ is equivalent  to $\neg P \eller \neg\neg P$. To see this, notice first that whenever $[x]\apt[y]$ for some $x,y\in\Z_{>0}$, we have $\neg(x\rho_P y)$ and thus
\[\neg(P \,\wedge\, (x-y)\in 2\Z) \,\eller\, \neg(\neg P \,\wedge\,(x-y)\in3\Z) \]
holds.
Now, let $m,n\in\Z_{>0}$, such that $[m]\apt[n]$ in $S_P$. By the division algorithm, $m$ is congruent modulo three to either $n$, or $n+1$ or $n+2$.
If $m\equiv n$ then $\neg\neg P$ holds, since $\neg P$ implies $m\rho_Pn$. 
Assume instead that $m\equiv n+1$. Then $m-(n-2)\in 3\Z$ and $n-(n-2)\in2\Z$, and hence
\[[m]\apt[n] \:\impl\: [m]\apt[n-2] \eller [n-2]\apt [n] \:\impl\: 
\neg(m\rho_P (n-2)) \eller \neg((n-2)\rho_P n) \:\impl\: 
(\neg\neg P) \eller (\neg P)\,.
\]
Last, assume that $m\equiv n+2$ modulo $3$. Then $m-(n+2)\in3\Z$ and $n-(n+2)\in2\Z$
whence, similarly, 
\[[m]\apt[n] \:\impl\:  [m]\apt[n+2] \eller [n+2]\apt [n] \:\impl\: (\neg\neg P)
\eller (\neg P) \,.\]
Conversely, if $\neg\neg P$ holds then $m\rho_Pn$ only if $(m-n)\in 2\Z$, and one readily verifies that the denial inequality $\neg\rho_P$ defines an apartness relation on $S_P$. If $\neg P$ holds then $m\rho_Pn \:\Leftrightarrow (m-n)\in2\Z$, and again $\neg\rho_P$ gives an apartness relation on $S_P$.

The group $S_P/{\equiv_6}$, where $\equiv_6$ denotes congruence modulo six, is an example  of a monogenic semigroup in which periodic elements exist, whilst existence of a smallest period implies decidability of the proposition $P$.
\end{ex}

%%%%%%%%%%%%%%%%%%%%%%%%%%%%%%%%%%%%%%%%
\subsection{Idempotents and apartness} \label{idempotents}

Here we collect some observations about the set $E(S)$ of idempotents in a semigroup $S$ with apartness.
Let $C = \{x\in S \mid x\apt x^2\}\subset S$.

\begin{lma}\label{idplma}
  The subset $C\subset S$ is strongly extensional, $C\subset\tild E(S)$, and  $\tild\, C = \neg C = E(S/{\approx}) \subset S$. 
\end{lma}

\begin{proof}
  Let $\phi:S\to S\times S,\,x\mapsto (x,x^2)$. Clearly, $\phi$ is a strongly extensional map. Now $C = \phi^{-1}(\apt)\subset S\times S$, and since $(\apt)\subset S\times S$ is a strongly extensional subset, so is $C\subset S$.

  Since $C\subset S$ is strongly extensional, the identity $\tild C = \neg C$ holds by Proposition~\ref{senegcompl}\eqref{senegcompl1}. Moreover,
  \[\neg C = \{x\in S \mid \neg(x\apt x^2)\} = \{x\in S \mid x\approx x^2\} = E(S/{\approx}),\]
and thus also
  $C\subset\tild\tild C = \tild E(S/{\approx}) \subset \tild E(S)$.
\end{proof}

As in the classical case, one proves that an element $s\in S$ is regular if and only if the set $V(s)$ of inverses of $s$ in inhabited.
The following is a co-congruence analogue of Lallement's lemma.

\begin{lma} \label{co-lallement}
  Let $\kappa$ be a co-congruence on $S$, and $a,x\in S$ such that $a^2$ is regular and $x\in V(a^2)$. If $a\kappa(axa)$ then $a\kappa a^2$.
\end{lma}

\begin{proof}
Assume that $a\kappa(axa)$. Then, by cotransitivity and co-compatibility,
\[
\begin{array}{cr@{}lcr@{}lcr@{}lcc}
    &(axa)&\kappa(a^2xa) &\eller& (a^2xa)&\kappa(a^2xa^2)=a^2 &\eller& a^2&\kappa a \\
  \Rightarrow &a&\kappa a^2 &\eller& a&\kappa a^2 &\eller& a^2&\kappa a 
  &\quad\Leftrightarrow\quad & a\kappa a^2 \,. \end{array}
\]
\end{proof}

\begin{prop}\label{idpprop}
  \begin{enumerate}
\item
  \label{idpprop1}
    If the apartness on $S$ is tight, then
    \[\tild\tild E(S) = \neg\neg E(S) = \neg\tild E(S) = E(S) \,.\]
  \item
    If $a\in S$ is periodic and $a\in\tild E(S)$, then $a\in C$.
    \label{idpprop2}
  \item
    Let $a,x\in S$ such that $a^2$ is regular and $x\in V(a^2)$. Then $a\in\tild E(S) \:\equ\: a\apt(axa) \:\equ\: a\apt a^2$.%
        \label{idpprop3}
  \item
    If $S$ is a periodic or a regular semigroup, then $\tild E(S) = C$ and hence $E(S)\subset S$ is closed in the apartness topology.
\label{idpprop4}
  \end{enumerate}
\end{prop}

\begin{proof}
  \eqref{idpprop1}
  This follows from Lemma~\ref{idplma}. Tightness means that $S/{\approx} = S$, thus $\tild C = E(S)$ and $\tild\tild E(S) = \tild\tild\tild C = \tild C = E(S)$. Negating the chain of inclusions $C\subset\tild E(S)\subset \neg E(S)$ gives $\neg\neg E(S) \subset\neg\tild E(S)\subset \neg C = E(S)$, whence $\neg\neg E(S) = \neg\tild C = E(S)$. 

\eqref{idpprop2}
By Proposition~\ref{periodicprop}\eqref{periodicprop5}, there exists an integer $n\ge 2$ such that $a^n \in E(S)$. Since $a\in\tild E(S)$, it follows that $a\apt a^n$. By cotransitivity, we have $a\apt a^2$ or $a^2\apt a^n$, which implies $a\apt a^2$ or $a\apt a^{n-1}$. The result follows by induction on $n\ge2$.

\eqref{idpprop3}
The implication $a\in\tild E(S) \:\impl\: a\apt(axa)$ is immediate, since $axa = axa^2xa = (axa)^2$ is an idempotent. The implication $a\apt a^2 \:\impl a\in\tild E(S)$ holds by Lemma~\ref{idplma}, and $a\apt(axa)\:\impl\: a\apt a^2$ by Lemma~\ref{co-lallement}.

\eqref{idpprop4}
The identity $\tild E(S) = C$ is immediate from \eqref{idpprop3} in the regular case, and from \eqref{idpprop2} and Lemma~\ref{idplma} when $S$ is periodic.
Since $C\subset S$ is open in the apartness topology by Lemma~\ref{idplma}, this means that $E(S)\subset S$ is closed.
\end{proof}

\begin{rmk}
The identity $C=\tild E(S)$ is not valid in general (even in the case of tight apartness).
For example, let $S$ be any semigroup with $E(S)=\emptyset$, and $T=S^1$. Then $\tild E(T) = S\subset T$, while it is not necessarily true that $a\apt a^2$ for all $a\in S$. A simple (albeit somewhat artifical) weak counterexample with tight apartness is $S = (\Z_{>0},+)$ with apartness given by $m\apt n \:\equ\: (m\ne n) \och (P\eller \neg P)$ for some proposition $P$.
\end{rmk}

%%%%%%%%%%%%%%%%%%%%%%%%%%%%%%%%%%%%%%%%
\subsection{Rees congruence} \label{rees}
The \emph{Rees congruence} $\rho_I =\{(x,y)\in S\times S \mid x=y \eller (x,y\in I)\}$ associated with an ideal $I\subset S$ gives a factor semigroup $S/\rho_I$ with zero element $I$. To carry out this construction for semigroups with apartness, one needs to define an apartness relation on $S/\rho_I$. 
To this end, it is natural to start from a \emph{co-ideal} $A\subset S$, giving rise to an ideal $\neg A$, and then define apartness on the factor semigroup $S/\rho_{\neg A}$. 

For any subset $A\subset S$, define a relation $\kappa_A$ on $S$ by 
\[ x\kappa_Ay \:\equ\: (\,(x\apt y) \wedge (x\in A \eller y\in A)\,) . \]

\begin{lma} \label{reeslma}
  If $A\subset S$ is strongly extensional then $\kappa_A$ is a coequivalence. 
\end{lma}

\begin{proof}
First, notice that
\begin{align}
 x\kappa_Ay \:&\equ\: (x\apt y\,\wedge\,x\in A) \eller (x\apt y\,\wedge\,y\in A)
\label{kappadef} \\
\intertext{and, since $A$ is strongly extensional, that}
x\in A \:&\equ\: \forall_{z\in S}(z\in A\eller z\apt x) \,\wedge\,x\in A \,.
\label{kappase}
\end{align}

Clearly, the relation $\kappa_A$ is strongly irreflexive and symmetric. 
For cotransitivity, let $(x,y)\in\kappa_A$. Then, by \eqref{kappadef}, without loss of
generality, we may assume that $x\apt y$ and $x\in A$. 
Now, 
\begin{align*}
  &x\apt y \,\wedge\, x\in A \\
  \;\equ\;
  &\forall_{z\in S}(x\apt z \eller z\apt y) \,\wedge\, x\in A \\
  \;\equ\;
  &\forall_{z\in S}[\,(x\apt z\wedge x\in A) \eller (z\apt y \wedge x\in A) \,] \\
  \;\stackrel{\eqref{kappase}}{\equ}\;
  &\forall_{z\in S}[\,(x\apt z\,\wedge\, x\in A) \eller 
    (z\apt y \,\wedge\, x\in A\,\wedge\,(z\in A\eller z\apt x)) \,] \\
  \;\equ\;
  &\forall_{z\in S}[\,(x\apt z\,\wedge\, x\in A) \eller 
    (z\apt y \,\wedge\, x\in A \,\wedge\, z\in A) \eller (z\apt y \,\wedge\, x\in A \,\wedge\, z\apt x) \,] \\
  \;\impl\;
  &\forall_{z\in S}[\, x\kappa_Az \eller z\kappa_Ay \eller z\kappa_Ax \,]
  \;\equ\; 
  \forall_{z\in S}[\, x\kappa_Az \eller z\kappa_Ay\,]
\end{align*}
which proves that $\kappa_A$ is cotransitive, and hence a coequivalence.
\end{proof}

\begin{prop} \label{reesprop}
  Let $A\subset S$ be a co-ideal. Then the following statements hold:
  \begin{enumerate}
  \item $\neg A=\tild A\subset S$ is an ideal; \label{reesprop1}
  \item $\kappa_A$ is a co-congruence; \label{reesprop2}
  \item $\rho_{\neg A}\subset \neg\kappa_A$. \label{reesprop3}
  \end{enumerate}
\end{prop}

\begin{proof}
\ref{reesprop1})
The identity $\neg A=\tild A$ holds by Lemma~\ref{senegcompl}\eqref{senegcompl1}. 
Let $x\in\neg A$ and $y\in S$. As $A\subset S$ is strongly extensional, for any $a\in A$ either $xy\apt a$ or $xy\in A$ holds. But as $A$ is convex, the latter condition implies that $x\in A$, contradicting the assumption that $x\in\neg A$. 
Therefore, $xy\apt a$ for all $a\in A$, that is, $xy\in\neg A$.
Similarly, one proves that $yx\in\neg A$. 

\ref{reesprop2})
The relation $\kappa_A$ is a coequivalence by Lemma~\ref{reeslma}. 
To establish co-compatibility, let $a,b,x,y\in S$ be such that $(ax)\kappa_A(by)$. By
definition, this means that $ax\apt by$, and either $ax\in A$ or $by\in A$.
Since multiplication in $S$ is strongly extensional, we have $a\apt b$ or $x\apt y$. 
By the convexity of $A$, if $ax\in A$ then $a,x\in A$, and if $by\in A$ then $b,y\in A$.
Hence, at least one of the following four statements holds:
\[a\apt b \;\mbox{and}\;a,x\in A,\quad a\apt b \;\mbox{and}\;b,y\in A,\quad
x\apt y \;\mbox{and}\;a,x\in A,\quad x\apt y \;\mbox{and}\;b,y\in A.\]
In the first two cases we have $a\kappa_Ab$, in the latter two, $x\kappa_A y$ holds. 

\ref{reesprop3})
Let $(x,y)\in\rho_{\neg A}$, so that either $x=y$ or 
$x,y\in \neg A$. If $x=y$ then $x\apt y$ is impossible; if $x,y\in \neg A$ then, by
definition, neither $x\in A$ nor $y\in A$ holds. In either case,
$(x\apt y) \,\wedge\, (x\in A \eller y\in A)$ is impossible, and hence 
$(x,y)\in\neg\kappa_A$.
\end{proof}

\begin{cor} \label{reescor}
Let $A\subset S$ be a co-ideal. Then $(S/\rho_{\neg A}, \kappa_A)$ is a semigroup with
apartness. 
\end{cor}

\begin{proof}
By Proposition~\ref{reesprop}(3), we have $\rho_{\neg A}\subset \neg\kappa_A$ and hence $\rho_{\neg A}\cap \kappa_A=\emptyset$. 
Thus, the coequivalence $\kappa_A$ defines and apartness relation on the factor set $S/\rho_{\neg A}$, by Proposition~\ref{apartnessonquotient}.
From Proposition~\ref{reesprop}(2) we get that $\kappa_A$ is a co-congruence, which implies that the multiplication is strongly extensional with respect to $\kappa_A$. 
\end{proof}
 
\begin{rmk}
While true classically, the inclusion $\rho_{\neg A}\supset \neg\kappa_A$ cannot be proved in general in a constructive framework, not even when the apartness is tight.
Counterexamples with non-tight apartness are easy to find;
\eg, $A=S$ gives $\neg\kappa_S=(\approx)$, whilst $\rho_{\neg S}=(=)$.

For a weak counterexample in a semigroup with tight apartness, consider the monoid
$S=\{0,1\}^\N$ of binary sequences with pointwise multiplication, and apartness defined by 
$x\apt y \:\equ\: \exists_{t\in\N}(x_t\ne y_t)$. 
Given any $z\in S$, let $A_z=\{x\in S\mid x\apt 0 \:\wedge\: z=1\}\subset S$, 
where $1\in S$ denotes the identity element. One readily verifies that $A_z$ is a co-ideal in $S$.
Now, 
$z\kappa_{A_z}1 \:\equ\: (z\apt 1)\,\wedge (z=1)$, so $(z,1)\in\neg\kappa_{A_z}$.
On the other hand,
\[
(z,1)\in\rho_{\neg A_z} \;\equ\; (z=1)\,\eller\,( z\in \neg A_z \,\wedge\,1\in \neg A_z )
\;\impl\; 
(z=1)\,\eller\,(1\in \neg A_z)
\;\equ\; 
(z=1)\eller\neg(z=1)
\]
which is not contructively provable in general.%
\footnote{The statement that, for each $z\in \{0,1\}^{\N}$, either $z=1$ or $\neg(z=1)$
  holds, is know as the \emph{weak limited principle of omniscience}.}
Hence we cannot prove that $\neg\kappa_A\subset\rho_{\neg A}$ holds for all co-ideals
$A\subset S$.
 
This means that the Rees construction gives rise to semigroups with apartness which is not necessarily tight, also in cases where the apartness on the original semigroup is tight.
\end{rmk}

%%%%%%%%%%%%%%%%%%%%%%%%%%%%%%%%%%%%%%%%%%%%%%%%%%%%%%%%%%%%%%%%%%%%%%%%%%%%%%%%
\section{Green's relations} \label{greensrel}

%%%%%%%%%%%%%%%%%%%%%%%%%%%%%%%%%%%%%%%%
\subsection{Constructivisation of classical results} \label{greenclassical}

The fundamental theory of Green's relations goes through with minimal modifications in the constructive setting. Below, we summarise the results, including proofs only in the cases where special consideration is required.

Throughout, $S$ is a semigroup.

\begin{dfn}
  Let $a,b\in S$. 
  \begin{enumerate}
  \item $a\le_L b \: :\equ \: a\in S^1 b, \quad
    a\le_R b \: :\equ \: a\in bS^1, \quad
    a\le_J b \: :\equ \: a\in S^1bS^1$.
  \item $a\Lrel b \: :\equ \: (a\le_L b) \:\wedge\: (b\le_L a),\quad
    a\Rrel b \: :\equ \: (a\le_R b) \:\wedge\: (b\le_R a),\quad
    a\Jrel b \: :\equ \: (a\le_J b) \:\wedge\: (b\le_J a)$.
  \item $\Hrel = \Lrel \cap \Rrel,\qquad \Drel = (\Lrel\cup\Rrel)^\infty$.
  \end{enumerate}
\end{dfn}
By construction, $\Lrel$, $\Rrel$, $\Jrel$, $\Hrel$ and $\Drel$ are equivalence relations; moreover, $\Lrel$ is a right congruence and $\Rrel$ is a left congruence.
Given $a\in S$, let $L_a = a\Lrel$, $R_a = a\Rrel$, etc.
Some basic results about these relations are summarised in Proposition~\ref{basicgreen} below.

\begin{prop} \label{basicgreen}
  Let $a,b\in S$.
  \begin{enumerate}
  \item $\Lrel \circ\Rrel = \Rrel\circ\Lrel$;
  \item $\Drel = \Lrel\circ\Rrel$;
  \item $a\Drel b \quad\equ\quad L_a\cap R_b \mbox{ is inhabited} \quad\equ\quad
    R_a\cap L_b \mbox{ is inhabited}$;
  \item $\Drel \subset \Jrel $;
  \item $S/\Lrel$, $S/\Rrel$, $S/\Jrel$ are posets under the induced order relations: $\le_L$, $\le_R$ respectively $\le_J$.
  \end{enumerate}
\end{prop}

\begin{prop}
  If $S$ is periodic then $\Drel = \Jrel$. 
\end{prop}

The next result, Green's lemma, similarly presents no additional obstacle from a constructive viewpoint.

\begin{lma}[Green's lemma] \label{greenslma}
  Let $a,b,s,s'\in S$, and assume that $as=b$, $bs'=a$ (so that $a\Rrel b$). Then the following hold.
  \begin{enumerate}
  \item The map $\rho_s$ induces an invertible map $L_a\to L_b,\:x\mapsto xs$, with inverse induced by $\rho_{s'}$;
  \item $\forall_{x\in L_a}: x\Rrel\rho_s(x)$;
  \item hence, $\rho_s$ and $\rho_{s'}$ induce mutually inverse maps between $H_a$ and $H_b$.
  \end{enumerate}
\end{lma}
In particular, whenever $a\Drel b$, there exists an invertible map $\phi:H_a\to H_b$ of the form $\phi = \rho_s\lambda_t$, where $s,t\in S$. Moreover, if $a,b\in S$ are such that $ab\in H_a$, then $\rho_b$ induces an invertible map $H_a\to H_a$. 

\begin{thm}[Green's theorem] \label{greensthm}
  Let $H$ be an $\Hrel$-class in $S$. If $H\cap H^2$ is inhabited, then $H$ is a group. 
\end{thm}

\begin{proof}
  Assuming that $H\cap H^2$ is inhabited, there exist $a,b\in H$ such that $ab\in H$. By Lemma~\ref{greenslma}, $\rho_b$ and $\lambda_a$ induce invertible maps $H\to H$. In particular, for all $h\in H$, we have $hb, ah\in H$. Lemma~\ref{greenslma} now implies that $\rho_h,\lambda_h:H\to H$ are invertible (and in particular surjective) and from Lemma~\ref{groupcondition} follows that $H$ is a group. 
\end{proof}

\begin{rmk}
In the classical context, Green's theorem is usually formulated as a disjunction: either $H\cap H^2 = \emptyset$, or $H$ is a group.
This is, of course, not possible for us. For a weak counterexample, let $P$ be a proposition, $S=\{0,a\}$ a two element semigroup under zero multiplication,
$\rho_P = \{(x,y)\in S\times S \mid (x=y)\eller  P\}$ and $H=H_{[a]}\subset S/\rho_P$.
Then $H\cap H^2 = \emptyset \:\equ\:\neg P$, whilst $H$ being a group is equivalent to $P$.
Similarly, in our formulation of Green's theorem, it is necessary to assume that $H\cap H^2$ is inhabited, not merely non-empty: in general, it is not possible to prove that $\neg\neg P \impl P$ holds.
\end{rmk}

\begin{prop} \label{regularprop}
  \begin{enumerate}
  \item If $a\in S$ is regular then all elements in $D_a$ is regular.
  \item In a regular $\Drel$-class, every $\Lrel$-class and every $\Rrel$-class contains an idempotent. 
  \item Every idempotent $e\in S$ is a left identity in $R_e$ and a right identity in $L_e$. 
  \end{enumerate}
\end{prop}

\begin{thm} \label{regularthm}
  Let $D\subset S$ be a regular $\Drel$-class, and $a\in D$.
  \begin{enumerate}
  \item $V(a)\subset D$.
    \label{regularthm_a}
  \item If $a'\in V(a)$ then $aa'\in R_a\cap L_{a'}$ and $a'a\in L_a\cap R_{a'}$ (and these elements are idempotents).%
    \label{regularthm_b}
  \item Let $e\in R_a$ and $f\in L_a$ be idempotents. Then there exists an element $a^*\in V(a)\cap L_e\cap R_f$ such that $aa^*=e$ and $a^*a=f$.
    \label{regularthm_c}
  \item If $a',a^*\in V(a)$ belong to the same $\Hrel$-class, then $a'=a^*$.
    \label{regularthm_d}
  \item Let $e,f\in S$ be idempotents. Then $e\Drel f$ if and only if there exist $a,a'\in S$ such that $a'\in V(a)$ and $aa'=e$, $a'a=f$. 
  \end{enumerate}
\end{thm}

\begin{proof}[Sketch of proof]
  \eqref{regularthm_c}
  Since $D$ is regular, we have $a=ata\in aS\cap Sa$ for some $t\in S$, and thus $S^1a=Sa$ and $aS^1=aS$. 
Now $e\in R_a\subset aS$, so $ax=e$ for some $x\in S$. Similarly, $f=ya$ for some $y\in S$. 
Set $a^*=fxe = yaxax = ye$.
Straightforward computations show that $aa^*=e$, $a^*a=f$, $a^*\in V(a)$, $e\Lrel a^*$ and $f\Rrel a^*$.

\eqref{regularthm_d}
If $a',a^*\in V(a)$ belong to the same $\Hrel$-class, then $aa'$ and $aa^*$ are idempotents in the $\Hrel$-class $H = L_a\cap R_{a^*} = L_a\cap R_{a'}$, whence $aa'= aa^*$ by Theorem~\ref{greensthm}. Similarly, $a'a = a^*a$.
It follows that $a^* = a^*aa^* = a^*aa'=a'aa'=a$. 
\end{proof}

\begin{cor}
  Let $a,b\in S$ such that $a\Drel b$.
  \begin{enumerate}
  \item If $H_a$ and $H_b$ are groups then there exist $c\in S$ and $c'\in V(c)$ such that $\rho_c\lambda_{c'}:H_a\to H_b$ is an isomorphism. 
  \item $ab\in R_a\cap L_b \;\equ\; L_a\cap R_b \mbox{ contains an idempotent.}$ 
  \end{enumerate}
\end{cor}

\begin{prop}
  Let $S$ be regular, and $a,b\in S$. Then
  \begin{align*}
    a\Lrel b \:&\equ\: \exists a'\in V(a),\, b'\in V(b): \, a'a = b'b \,,\\
    a\Rrel b \:&\equ\: \exists a'\in V(a),\, b'\in V(b): \, aa' = bb' \,,\\
    a\Hrel b \:&\equ\: \exists a'\in V(a),\, b'\in V(b): \,(\, a'a = b'b \,\wedge\, aa'=bb'\,) \,.
  \end{align*}
\end{prop}

\begin{prop} Let $U\subset S$ be a regular subsemigroup. Then the following hold.
  \begin{enumerate}
  \item $(\le_L^U) = (\le_L^S)\cap (U\times U),\quad (\le_R^U) = (\le_R^S)\cap (U\times U)$;
  \item $\Lrel^U = \Lrel^S\cap(U\times U),\quad
    \Rrel^U = \Rrel^S\cap(U\times U),\quad
    \Hrel^U = \Hrel^S\cap(U\times U)$.
  \end{enumerate}
\end{prop}

Clearly, the inclusions ``$\subset$'' are true for any subsemigroup $U\subset S$.
The identities $\Drel^U = \Drel^S\cap(U\times U)$ and $\Jrel^U = \Jrel^S\cap(U\times U)$ are not true in general a regular subsemigroup $U\subset S$.

\begin{proof}
Clearly, the second part of the proposition follow directly from the first.
Let $a,b\in U$, and assume that $a\le_L^S b$. As $U$ is regular,
there exists inverses $a'$ and $b'$ in $U$ of $a$ and $b$ respectively. By Theorem~\ref{regularthm}\eqref{regularthm_b}, the relations $aa'\Lrel^U a$ and $bb'\Lrel^U b$, and thus also $aa'\Lrel^S a$ and $bb'\Lrel^S b$, hold. Hence $aa'\le_L^S bb'$, that is, $aa'\in Sbb'$. As $bb'$ is a right identity element in $Sbb'$, we have $aa'= (aa')(bb')\in Ubb'$, so $aa'\le_L^U bb'$. Combined with the relations $aa'\Lrel^U a$ and $bb'\Lrel^U b$, this implies that $a\le_L^U b$.

From the above, we conclude that $(\le_L^U) = (\le_L^S)\cap (U\times U)$.
The proof of the identity $(\le_R^U) = (\le_R^S)\cap (U\times U)$ is completely analogous.
\end{proof}

\begin{lma}[Lallement's lemma]
  Let $\rho$ be a congruence on a regular semigroup $S$, and $[a]\in S/\rho$ an idempotent. Then there exists an idempotent $e\in S$ such that $[e]=[a]$, and $e\le_L a$, $e\le_R a$.
\end{lma}

A congruence $\rho$ on $S$ is \emph{idempotent separating} if $e\rho f$ implies $e=f$ for all idempotents $e,f\in S$. 

\begin{prop}
  Let $\rho$ be a congruence on a regular semigroup $S$. Then $\rho$ is idempotent separating if and only if $\rho\subset\Hrel$. 
\end{prop}

%%%%%%%%%%%%%%%%%%%%%%%%%%%%%%%%%%%%%%%%
\subsection{Constructive friends of Green's relations} \label{greenfriends}

In this section, we shall define constructive friends of Green's relations, and derive some basic properties of these. Throughout, $S$ is a semigroup with apartness.

We will need the following assumption on the logic in $S$:
Let $P(x)$ be a predicate of the form $u(x)\apt v$ or $\forall_s(u(s,x)\apt v)$ where $u$ and $v$ are terms and $x$ does not occur in $v$.
Let $Q$ be a predicate of the form $u\apt v$, $\forall_s(u(s)\apt v)$, or $\forall_{s,t}(u(s,t)\apt v)$, where $u$ and $v$ are terms, $x$ does not occur in either $u$ or $v$, and $s$ and $t$ do not occur in $v$.
Then $S$ satisfies the \emph{constant comains principle} 
\begin{equation} \label{cdinS}
  \forall x(P(x) \eller Q) \;\impl\; \forall x\,P(x) \eller Q 
\end{equation}
for $P(x)$ and $Q$.

\begin{lma} \label{seideals}
  \begin{enumerate}
  \item For every $b\in S$, the subsets $Sb$, $bS$, $SbS$, $S^1b$, $bS^1$ and $S^1bS^1$ of $S$ are closed in the apartness topology. \label{seideals1}
  \item For every $a\in S$, the sets
    \begin{align*}
    {}_aD^l &= \{ b\in S\mid a\in\tild\,(Sb)\} , &
    {}_aD^r &=\{ b\in S\mid a\in\tild\,(bS)\} , &
    {}_aD^j &=\{ b\in S\mid a\in\tild\,(SbS)\}, \\
    {}_a\tilde{D}^l &= \{ b\in S\mid a\in\tild\,(S^1b)\} , &
    {}_a\tilde{D}^r &= \{ b\in S\mid a\in\tild\,(bS^1)\},\;\mbox{and} &
    {}_a\tilde{D}^j &= \{ b\in S\mid a\in\tild\,(S^1bS^1)\}
    \end{align*}
    are strongly extensional subsets of $S$. \label{seideals2}
  \end{enumerate}
\end{lma}

\begin{proof}
\eqref{seideals1}
Let $a\in \tild\,(Sb)$ and $x\in S$. Then $a\apt sb$ for all $s\in S$ and hence, by cotransitivity, $a\apt x$ or $x\apt sb$. By \eqref{cdinS}, we can now infer that either $a\apt x$ or $\forall_{s\in S}(x\apt sb)$ holds, so either $a\apt x$ or $x\in\tild\,(Sb)$. This proves that $\tild\,(Sb)\subset S$ is strongly extensional, that is, $Sb\subset S$ is closed.
As the collection of closed subsets of $S$ is closed under finite unions, and $\{b\}\subset S$ is closed by Proposition~\ref{separation}\eqref{separation1}, it follows that $S^1b = \{b\}\cup Sb\subset S$ is closed, too. 
The proof for right ideals is completely dual, and the two-sided case is analogous. 

\eqref{seideals2}
This is similar to \eqref{seideals1}. We prove the claim for the sets ${}_aD^l$ and ${}_a\tilde{D}^l$ only. 
Let $b\in {}_aD^l$ and $c\in S$. Then, for all $s\in S$, we have $a\apt sb$ and hence $a\apt sc$ or $sc\apt sb$ by cotransitivity. The latter implies $c\apt b$, and thus we have proved
$\forall_{s\in S}(a\apt sc \eller c\apt b)$.
By \eqref{cdinS}, this is equivalent to $\forall_{s\in S}(a\apt sc ) \eller c\apt b$, that is, $a\in\tild\,(Sb)$ or $c\apt b$. Hence, ${}_aD^l\subset S$ is a strongly extensional subset, and it follows that so is $ {}_a\tilde{D}^l= {}_aD^l\cap\tild\,\{a\}$.
\end{proof}

The following consequence of Lemma~\ref{seideals}\eqref{seideals1} may be of independent interest. 

\begin{prop} \label{sefgideals}
 Every finitely generated (left/right/two-sided) ideal in $S$ is closed in the apartness topology. 
\end{prop}

\begin{proof}
  As a finite union of closed sets is closed, this follows from Lemma~\ref{seideals}\eqref{seideals1}. 
\end{proof}

\begin{dfn}[Green co-quasiorder relations]
  Let $a,b\in S$. Define relations $\succ_l$, $\succ_r$ and $\succ_j$ on $S$ by 
  \[ a\succ_lb \:\equ\: a\in\tild\,(S^1b), \quad
  a\succ_rb \:\equ\: a\in\tild\,(bS^1), \quad
  a\succ_jb \:\equ\: a\in\tild\,(S^1bS^1). \]
\end{dfn}

\begin{prop} \label{greencoordersse}
 The relations $\succ_l$, $\succ_r$ and $\succ_j$ are strongly extensional subsets of $S\times S$.
\end{prop}

\begin{proof}
By Lemma~\ref{seideals}\eqref{seideals1} and \eqref{seideals2}, the subsets $(a\!\succ_l)\subset S$ and $(\succ_l\!a)\subset S$ are strongly extensional for all $a\in S$. Proposition~\ref{serelations} now implies that $(\succ_l)\subset S\times S$ is strongly extensional. The same argument works for $\succ_r$ and $\succ_j$.
\end{proof}

\begin{rmk} \label{notsermk}
  Proposition~\ref{greencoordersse} is not provable without additional assumptions on the logic in $S$. Indeed, in the multiplicative monoid $\R$, $1\in\R b$ holds whenever $b\apt0$, and hence
  \[(1\succ_l) \subset \neg\,(\apt0) = \{0\}. \]
  As moreover $1\in\tild\,\{0\} = \tild\,(\R0)$, we have $1\succ_l0$, and thus $(1\succ_l) = \{0\}$.
  Now, if $(\succ_l)\subset \R\times\R$ is strongly extensional then so is $\{0\}=(1\succ_l)\subset\R$ by Proposition~\ref{serelations}. Hence, for all $x,y\in\R$, $x-y=0$ or $(x-y)\apt0$, that is, $x=y$ or $x\apt y$.
  This means that the $\R$ is discrete with respect to the usual apartness, a statement which is equivalent to the limited principle of omniscience (see for example \cite[Section~II.3]{mrr88}). 
\end{rmk}

In Proposition~\ref{coorderproperties} below, we collect some properties of the relations $\succ_l$, $\succ_r$ and $\succ_j$.
We will need the following observation.

\begin{lma} \label{complisconv}
 If $I\subset S$ is a left (right) ideal, then $\tild I\subset S$ is right (left) convex.
\end{lma}

\begin{proof}
For all $x,y\in S$:\;
$xy\in\tild I \:\equ\: \forall_{a\in I}(xy\apt a) \:\stackrel{[a=xb]}{\impl}\:    \forall_{b\in I}(xy\apt xb) \:\impl\: \forall_{b\in I}(y\apt b) \:\equ\: y\in \tild I$.
\end{proof}

\begin{prop} \label{coorderproperties}
  Let $a,b,c\in S$.
  \begin{enumerate}
      \item\label{coorderproperties1} The following identities hold:
    \[(\succ_l) = \tild\,(\le_L),\quad
  (\succ_r) = \tild\,(\le_R),\quad
  (\succ_j) = \tild\,(\le_J); \]
  \item
    the relation $\succ_l$ is right co-compatible, and $\succ_r$ is left co-compatible, with the multiplication in $S$;
    \label{coorderproperties2}
  \item\label{coorderproperties3}
    the subset $(\succ_la)=\tild\,(S^1a)\subset S$ is a left co-ideal, $(\succ_ra)=\tild\,(aS^1)\subset S$ is a right co-ideal, and $(\succ_ja)=\tild\,(S^1aS^1)\subset S$ is a co-ideal;
  \item
    if $A\subset S$ is a left (respectively right, two-sided) co-ideal and $a\notin A$, then $A\subset (\succ_la)$ (respectively $A\subset (\succ_ra)$, $A\subset (\succ_ja)$);
    \label{coorderproperties4}
  \item
    the relations $\succ_l$, $\succ_r$ and $\succ_j$ are co-quasiorders;
    \label{coorderproperties5}
  \item if $a\succ_lc$ and $a\le_Lb$ then $b\succ_lc$.
      \label{coorderproperties6}
  \end{enumerate}
\end{prop}

\begin{proof}
\eqref{coorderproperties1}
Let $a,b\in S$, and assume that $(a,b)\in\tild\,(\le_L)$. Then
\begin{align*} 
  \forall_{x,y\in S}(x\in S^1y \:\impl\: (a,b)\apt(x,y)) \quad 
  \stackrel{[y=b]}{\impl} \quad & \forall_{x\in S}(x\in S^1b \:\impl\: a\apt x)) \\
  \equ\quad\;\, & a\in \tild\,(S^1b) \;\equ\;  a\succ_l b,
\end{align*}
which proves the inclusion $\tild\,(\le_l) \subset (\succ_l)$.
On the other hand, by Proposition~\ref{greencoordersse}, $(\succ_l)\subset S\times S$ is strongly extensional, and from the definitions it is clear that $(\succ_l)\subset \neg(\le_L)$.
It follows that $(\succ_l)\subset \tild\,(\le_L)$, by Lemma~\ref{senegcompl}\eqref{senegcompl2}.

\eqref{coorderproperties2}
Let $x,y,z\in S$, and assume that $xz\succ_lyz$. Then, for all $s\in S^1$, we have $xz\apt syz$ and hence $x\apt sy$ since the multiplication in $S$ is strongly extensional. This means that $x\succ_ly$, so $\succ_l$ is right co-compatible. 

\eqref{coorderproperties3}
By Lemma~\ref{seideals}\eqref{seideals1}, the subsets $(\succ_la) = \tild\,(S^1a)$, $(\succ_ra) = \tild\,(aS^1)$ and $(\succ_ja) = \tild\,(S^1aS^1)$ of $S$ are strongly extensional.
Right/left/two-sided convexity follows from Lemma~\ref{complisconv}. 

\eqref{coorderproperties4}
Let $A\subset S$ be a left co-ideal, and $a\notin A$. If $sa\in A$ for some $s\in S^1$ then $a\in A$ by right convexity. Hence, $A\subset \neg(S^1a)$ and, as $A\subset S$ is strongly extensional, Lemma~\ref{senegcompl}\eqref{senegcompl2} gives that $A\subset\tild\,(S^1a)=(\succ_la)$. 

\eqref{coorderproperties5}
Strong irreflexivity is clear from the definitions.
Let $a,b,c\in S$ with $a\succ_l b$.
Observing that $(\succ_lb)\subset S$ is strongly extensional (by Lemma~\ref{seideals} \eqref{seideals1}) and right convex (by \eqref{coorderproperties3}), we get
\begin{align*}
  a\succ_lb \:\impl\:&\forall_{t\in S^1} (\,a\apt tc \eller (tc\succ_lb)\,) \:\impl\:
  \forall_{t\in S^1} (\,a\apt tc \eller (c\succ_lb)\,) \stackrel{\eqref{cdinS}}{\equ}\: \forall_{t\in S^1} (a\apt tc) \eller (c\succ_lb) \\
  \:\equ\:
  &(a\in \tild\,(S^1c)) \eller (c\succ_lb) \;\equ\; (a\succ_l c) \eller (c\succ_l b),
\end{align*}
so $\succ_l$ is cotransitive. 

\eqref{coorderproperties6}
Since $\succ_l$ is cotransitive (by \eqref{coorderproperties5}), $a\succ_lc$ implies that either $a\succ_lb$ or $b\succ_lc$ holds. But from $a\le_L b$ follows that $\neg(a\succ_lb)$, hence, $b\succ_lc$. 
\end{proof}

\begin{dfn} \label{def_cosandy}
  Set $\lrel = (\succ_l)\cup (\succ_l)^{-1}$,  $\rrel = (\succ_r)\cup (\succ_r)^{-1}$,  $\jrel = (\succ_j)\cup (\succ_j)^{-1}$, and $\hrel = \lrel\cup\rrel$. 
\end{dfn}

\begin{lma} \label{lrjd-coeq}
  The relations $\lrel$, $\rrel$, $\jrel$ and $\hrel$ are co-equivalences.
\end{lma}

\begin{proof}
  Observe that the union of a collection of strongly irreflexive/cotransitive relations is again strongly irreflexive/cotransitive.  
 As $\succ_l$, $\succ_r$ and $\succ_j$ are strongly irreflexive and cotransitive by Proposition~\ref{coorderproperties}\eqref{coorderproperties2}, it follows that so are $\lrel =(\succ_l)\cup (\succ_l)^{-1}$, $\rrel=(\succ_r)\cup (\succ_r)^{-1}$, $\jrel =(\succ_j)\cup (\succ_j)^{-1}$ and $\hrel = \lrel\cup\rrel$.
Symmetry is clear from the definitions.
\end{proof}

\begin{thm} \label{landrcommute}
  $\rrel*\lrel = \lrel*\rrel$
\end{thm}

\begin{proof}
For all $a,b\in S$, we have
\begin{align*}
  &(a,b)\in\rrel*\lrel \;\equ\; \forall_d(\,a\rrel d \eller d\lrel b\,) \\
  \equ\;
  &\forall_{d}(\,(\forall_v(dv\apt a)\vee \forall_u(d\apt au)) \:\eller\:
  (\forall_y(d\apt yb)\vee \forall_x(xd\apt b)) \,) \\
  \impl\;
  &\forall_{d,v,u,y,x}(dv\apt a \eller d\apt au \eller d\apt yb \eller xd\apt b) \\
  \impl\;
  &\forall_{c,v,u,y,x}(ycuv\apt a \eller ycu\apt au \eller ycu\apt yb \eller xycu\apt b)
  && \mbox{(setting $d=ycu$)}\\
  \impl\;
  &\forall_{c,v,u,y,x}(\,(ycuv\apt ybv \eller ybv\apt yc \eller yc\apt a) \eller ycu\apt au \\
  &\qquad\quad \eller ycu\apt yb \eller (xycu\apt xau \eller xau\apt cu \eller cu\apt b)\,)
  \quad&& \mbox{(by cotransitivity of $\apt$)}\\
  \impl\;
  &\forall_{c,v,u,y,x}(\,(cu\apt b \eller bv\apt c \eller yc\apt a) \eller yc\apt a
  && \mbox{(multiplication is}\\
  &\qquad\quad \eller cu\apt b \eller (yc\apt a \eller xa\apt c \eller cu\apt b)\,)
  && \mbox{\quad\: strongly extensional)}\\
  \equ\;
  &\forall_{c,v,u,y,x}(\,cu\apt b \eller bv\apt c \eller yc\apt a \eller xa\apt c \,)\\
  \impl\;
  &\forall_c(\, \forall_u(cu\apt b)\eller \forall_v(bv\apt c)\eller \forall_y(yc\apt a)\eller \forall_x(xa\apt c) \,)
  &&\mbox{(by \eqref{cdinS})} \\
  \equ\;
  &\forall_c(b\rrel c \eller c\lrel a)
  \;\equ\; (a,b)\in\lrel*\rrel \,.
\end{align*}
So $\rrel*\lrel\subset\lrel*\rrel$ and, by symmetry, $\lrel*\rrel\subset\rrel*\lrel$. Hence, $\rrel*\lrel =\lrel*\rrel$.
\end{proof}

\begin{cor} \label{discoequiv}
The relation $\lrel*\rrel$ is a coequivalence. It is the (unique) maximal cotransitive relation contained in $\lrel\cap\rrel$.
\end{cor}

\begin{proof}
By Lemma~\ref{*properties}(\ref{*properties2},\ref{*properties5}), a strongly irreflexive and symmetric relation $\kappa$ is a co-equivalence if and only if $\kappa*\kappa = \kappa$. 
The relations $\lrel$ and $\rrel$ are co-equivalences by Lemma~\ref{lrjd-coeq}. From Lemma~\ref{*properties}(\ref{*properties2},\ref{*properties3}) and Theorem~\ref{landrcommute}, it follows that $\lrel*\rrel = \rrel*\lrel$ is strongly irreflexive and symmetric. Moreover, using Theorem~\ref{landrcommute} together with associativity we get
\[ (\lrel*\rrel)*(\lrel*\rrel) = (\lrel*\lrel)*(\rrel*\rrel) = \lrel*\rrel \,.\]
Hence $\lrel*\rrel$ is a co-equivalence.

For any cotransitive relation $\kappa\subset\lrel\cap\rrel$, the inclusion $\kappa\subset\lrel*\rrel$ holds by Lemma~\ref{*properties}\eqref{*properties7}.
\end{proof}

\begin{dfn} \label{drel-def}
Set $\drel = \lrel*\rrel \subset S\times S$.
\end{dfn}

\begin{prop} \label{cogreenproperties}
  \begin{enumerate}
  \item The relations $\lrel$, $\rrel$, $\jrel$, $\hrel$ and $\drel$ are coequivalences;
    \label{cogreenproperties1}
  \item $\jrel \subset\drel \subset \lrel \subset \hrel$, \quad
    $\drel \subset \rrel \subset \hrel$;
\label{cogreenproperties2}
  \item $\lrel\subset \tild\Lrel$,\quad $\rrel\subset \tild\Rrel$,\quad $\jrel\subset \tild\Jrel$,\quad $\hrel\subset \tild\Hrel$, and\quad $\drel\subset \tild\Drel$;
    \label{cogreenproperties3}
  \item the relation $\lrel$ is a right co-congruence, and $\rrel$ is a left co-congruence, on $S$.
    \label{cogreenproperties4}
  \end{enumerate}
\end{prop}

\begin{proof}
\eqref{cogreenproperties1}
The relations $\lrel$, $\rrel$, $\jrel$ and $\hrel$ are symmetric by construction, and strongly irreflexive and cotransitive by Lemma~\ref{lrjd-coeq}. 
The relation $\drel$ is a coequivalence by Corollary~\ref{discoequiv}.

\eqref{cogreenproperties2}
From the definitions, it is clear that $\jrel\subset \lrel\subset\hrel$ and $\jrel\subset\rrel\subset\hrel$. Using Lemma~\ref{*properties}, we get that
$\jrel = \jrel*\jrel \subset \lrel*\rrel \subset \lrel\cap\rrel$. As $\drel = \lrel*\rrel$, this establishes remaining inclusions.

\eqref{cogreenproperties3}
First, by Proposition~\ref{coorderproperties}\eqref{coorderproperties1}
\[ \lrel=(\succ_l)\cup(\succ_l)^{-1} = \left(\tild\le_L)\cup (\tild\le_L)^{-1}\right) \subset \tild\,\left((\le_L)\cap(\le_L)^{-1}\right) = \tild\Lrel . \]
Similarly, $\rrel\subset\tild\Rrel$ and $\jrel\subset\tild\Jrel$ hold.
From the above now follows that
\[ \hrel = \lrel\cup\rrel \subset \tild\Lrel \cup\tild\Rrel \subset \tild\,(\Lrel\cap\Rrel) = \tild\Hrel.\]
Last, using Lemma~\ref{*properties}\eqref{*properties1}, we get
\[ \drel = \lrel*\rrel \subset (\tild\Lrel)*(\tild\Rrel) \subset (\neg\Lrel)*(\neg\Rrel) \subset \neg(\Lrel\circ \Rrel) \,. \]
By \eqref{cogreenproperties1}, the relation $\drel$ is cotransitive, and hence strongly extensional by Proposition~\ref{kappaisse}.
Thus, Lemma~\ref{senegcompl}\eqref{senegcompl2} implies that $\drel\subset \tild\,(\Lrel\circ \Rrel) = \tild{\Drel}$. 

\eqref{cogreenproperties4}
By Proposition~\ref{coorderproperties}\eqref{coorderproperties2}, the relation $\succ_l$, and thus also $(\succ_l)^{-1}$, is right co-compatible with the multiplication in $S$. It follows that the same is true for $\lrel= (\succ_l)\cup (\succ_l)^{-1}$.
Since $\lrel$ is a coequivalence by \eqref{cogreenproperties1}, this means that it is a right co-congruence. 
\end{proof}

\begin{thm} \label{periodicthm}
  If $S$ is periodic, then $\drel = \jrel$. 
\end{thm}

\begin{lma} \label{periodiclma}
  Let $a,b\in S$, $x,y,u,v\in S^1$, and $m\in\Z_{>0}$.
  \begin{enumerate}
  \item If $a\apt x^may^m$ then $a\apt xay$.
    \label{periodiclma1}
  \item If $a\apt(ux)^ma(yv)^m$ then $a\apt ubv$ or $b\apt xay$.
    \label{periodiclma2}
  \end{enumerate}
\end{lma}

\begin{proof}
\eqref{periodiclma1}
Induction on $m$: The case $m=1$ is clear. Assume that $a\apt x^may^m \:\impl\: a\apt xay$ holds.
Then $a\apt x^{m+1}ay^{m+1}$ implies that $a\apt xay$ or $xay\apt x^{m+1}(xay)y^{m+1}$.
In the latter case, we get $a\apt x^may^m$ and thus $a\apt xay$ by the induction hypothesis. 

\eqref{periodiclma2}
We have
\[
a\apt(ux)^ma(yv)^m \;\stackrel{\eqref{periodiclma1}}{\impl}\;
a\apt uxayv \;\impl\; a\apt ubv \eller ubv\apt uxayv \;\impl\; a\apt ubv \eller b\apt xay \,.
\]
\end{proof}

\begin{proof}[Proof of Theorem~\ref{periodicthm}]
The inclusion $\jrel\subset\drel$ holds by Proposition~\ref{cogreenproperties}\eqref{cogreenproperties2}.
For the converse, note that
\begin{equation} \label{DLR}
(a,b)\in\drel=\lrel*\rrel \:\equ\: \forall_{c\in S}(a\lrel c \eller c\rrel b)
\:\impl\: \forall_{x\in S^1}(a\lrel(xa) \eller (xa)\rrel b) \,.
\end{equation}
We shall analyse the two cases $a\lrel(xa)$ and $(xa)\rrel b$ separately below.

First,
\begin{align*}
  &a\lrel(xa) \:\equ\: \left(a\in \tild\,(S^1(xa)) \eller xa\in\tild\,(S^1a)\right) \:\equ\:
  a\in\tild\,(S^1(xa)) \:\impl\: \forall_{u\in S^1,\,m\in\Z_{>0}}\left(a\apt(ux)^ma\right) \\
  &\impl\:
  \forall_{y,u,v\in S^1,\,m\in\Z_{>0}} \\
  &\qquad\qquad\left(\, a\apt (ux)^ma(yv)^m \eller (ux)^ma(yv)^m \apt (ux)^{2m}a(yv)^m \eller (ux)^{2m}a(yv)^m \apt(ux)^ma \,\right) \\
  &\impl\:
  \forall_{y,u,v\in S^1,\,m\in\Z_{>0}} 
  \left(\, a\apt (ux)^ma(yv)^m \eller (ux)^m \apt (ux)^{2m} \eller (ux)^{m}a(yv)^m \apt a \,\right) \\
  &\equ\:
  \forall_{y,u,v\in S^1,\,m\in\Z_{>0}} 
  \left(\, a\apt (ux)^ma(yv)^m \eller (ux)^m \apt (ux)^{2m} \,\right) \,.
\end{align*}
By assumption, $S$ (and thus $S^1$) is periodic so, by Proposition~\ref{periodicprop}\eqref{periodicprop5}, for all $u,x\in S^1$ there exists an $n\in\Z_{>0}$ such that $(ux)^{2n}=(ux)^n$. Hence, we get
\begin{align} \label{aLxa}
  &a\lrel(xa) \:\impl\:
  \forall_{y,u,v\in S^1}\exists_{n\in\Z_{>0}} \left( a\apt (ux)^na(yv)^n \right)
  \:\impl\:
  \forall_{y,u,v\in S^1} \left( a\apt ubv \eller b\apt xay \right) 
\end{align}
by Lemma~\ref{periodiclma}\eqref{periodiclma2}.

Next, consider the case $(xa)\rrel b$. We have
\begin{align*} 
  &(xa)\rrel b \:\equ\: \left( b\in\tild\,((xa)S^1) \eller xa\in\tild\,(bS^1) \right)
  \:\impl\:
  \forall_{y,u,v\in S^1,\,m\in\Z_{>0}}\left(\, b\apt xay \eller xa\apt b(vy)^{m-1} \,\right) \\ %%%%%%%%
  &\impl\:
  \forall_{y,u,v\in S^1,\,m\in\Z_{>0}}
  (\, b\apt xay \eller xa\apt x(ux)^{m+1}a(yv)^{m+1} \\
  &\qquad\qquad
  \eller x(ux)^{m+1}a(yv)^{m+1}\apt (xu)^{m+1}b(vy)^mv \\
  &\qquad\qquad
  \eller (xu)^{m+1}b(vy)^mv \apt (xu)^{m+1}b(vy)^{2m}v \eller (xu)^{m+1}b(vy)^{2m}v \apt b(vy)^{m-1} \,) \\ %%%%%%%%
  &\impl\:
  \forall_{y,u,v\in S^1,\,m\in\Z_{>0}}
    (\, b\apt xay \eller a\apt (ux)^{m+1}a(yv)^{m+1} \eller xay\apt b \\
  &\qquad\qquad
  \eller (vy)^m \apt (vy)^{2m} \eller (xu)^{m+1}b(vy)^{m+1}v \apt b \,) \\ %%%%%%%%
  &\impl\:
  \forall_{y,u,v\in S^1,\,m\in\Z_{>0}}
  (\, b\apt xay \eller ( a\apt ubv \eller b\apt xay )
  \eller (vy)^m \apt (vy)^{2m} \eller ( b\apt xay \eller x\apt ubv ) \,)
\end{align*}
where the last implication uses Lemma~\ref{periodiclma}\eqref{periodiclma2} twice.
Similarly to the previous case, periodicity of $S$ implies the existence of a positive integer $n$ such that $(vy)^n = (vy)^{2n}$. Thus
\begin{equation} \label{xaRb}
  (xa)\rrel b \:\impl\:
  \forall_{y,u,v\in S^1} \exists_{n\in\Z_{>0}}
  (\, b\apt xay \eller a\apt ubv \,) \:\impl\:
  \forall_{y,u,v\in S^1} 
  (\, b\apt xay \eller a\apt ubv \,) \,.
\end{equation}

Finally, combining the equations \eqref{DLR}, \eqref{aLxa} and \eqref{xaRb}, we have
\begin{equation*}
  a\drel b \:\impl\:
  \forall_{x,y,u,v\in S^1} ( a\apt ubv \eller b\apt xay) \:\stackrel{\eqref{cdinS}}{\equ}\:
  \forall_{u,v\in S^1} (a\apt ubv) \eller \forall_{x,y\in S^1} (b\apt xay) \:\equ\:
  a\jrel b \,,
\end{equation*}
that is, $\drel\subset\jrel$. 
\end{proof}

The following result is an analogue of Green's lemma (Lemma~\ref{greenslma}).

\begin{lma} \label{cogreenlma}
  Let $as=b$ and $bt=a$ (so that $a\Rrel b$). Then
  \begin{enumerate}
  \item $\rho_s^{-1}(b\lrel)\subset a\lrel$ and $\rho_t^{-1}(a\lrel)\subset b\lrel$;
\label{cogreenlma1}
\item for all $x\in S$, if  $x\rrel(xs)$ then $x\succ_la$ (and consequently $x\lrel a$) holds.
  \label{cogreenlma2}
  \end{enumerate}
\end{lma}

\begin{proof}
\eqref{cogreenlma1}
If $x\in \rho_s^{-1}(b\lrel)$ then $(xs)\lrel b$. Since $b = as$, and $\lrel$ is a right congruence by Proposition~\ref{cogreenproperties}\eqref{cogreenproperties4}, this implies that $x\lrel a$, that is, $x\in a\lrel$.

\eqref{cogreenlma2}
Assume that $x\rrel(xs)$. Since $xs\in xS^1$, we have $x\apt xsr$ for all $r\in S^1$ so, in particular, $x\apt xst$. Thus,
\[\forall_{u\in S^1}(x\apt ua \eller ua\apt xst) \;\stackrel{a=ast}{\equ}\;
\forall_{u\in S^1}(x\apt ua \eller uast\apt xst) \;\impl\;
\forall_{u\in S^1}(x\apt ua) \;\equ\; x\succ_la \,. \]
\end{proof}

\begin{prop} \label{cogreenprop}
  Let $e\in S$ be an idempotent. Then $e\hrel \subset S$ is an co-subsemigroup. 
\end{prop}

\begin{proof}
The relation $\hrel$ is strongly extensional and thus, by Proposition~\ref{serelations}, 
$e\hrel$ is a strongly extensional subset of $S$.

If $x,y\in S$ are such that $xy\in e\hrel$, then either $e\lrel (xy)$ or $e\rrel (xy)$.
Assume that $e\rrel (xy)$. Cotransitivity of $\rrel$ (Proposition~\ref{cogreenproperties}\eqref{cogreenproperties1}) implies that either $e\rrel x$, $x\rrel xe$, or $xe\rrel xy$ holds.

Assume that $x\rrel xe$. Then Lemma~\ref{cogreenlma}\eqref{cogreenlma2}, with $a=b=s=t=e$, implies that $x\lrel e$.
On the other hand, if $(xe)\rrel(xy)$ then $e\rrel y$ by Proposition~\ref{cogreenproperties}\eqref{cogreenproperties4}. 

In sum, if $e\rrel (xy)$ then either $e\rrel x$, $x\lrel e$, or $e\rrel y$. 
Dually, if $xy\in e\lrel$ then either $e\lrel x$, $x\rrel e$ or $e\lrel y$.
This proves that $e\hrel = {e\lrel} \,\cup\, {e\rrel}\subset S$ is a co-subsemigroup. 
\end{proof}

\begin{rmk}
Proposition~\ref{cogreenprop} may be viewed as a partial analogue of Green's theorem, for the constructive friend $\hrel$ of $\Hrel$.
With Green's theorem stating that $H_e$ is a group for any idempotent $e\in S$, a full analogue for $\hrel$ would be that $e\hrel$ is a ``co-subgroup'' in $S/{\approx}$, that is, that $\tild\,(e\hrel)/{\approx}$ is a group.
However, this cannot be proved in general, as the following example shows.
\end{rmk}

\begin{ex}
Let $S= (\N \!\times\! \N \!\times\! \N,+)$ with component-wise addition. Given a proposition $P$, define $(\apt_P)\subset S\times S$ by
\begin{align*} (x_1,x_2,x_3)\apt(y_1,y_2,y_3) \equ{} &(P \och ((x_1+y_2\ne x_2+y_1) \eller x_3\ne y_3)\,) \\
  \eller &(\neg P \och ((x_1+y_3\ne x_3+y_1) \eller x_2\ne y_2)\,)
\end{align*}
(where ``$\ne$'' is the denial inequality on $\N$).
It is plain to see that $(S,\apt_P)$ is a semigroup with apartness, and identity element $0=(0,0,0)$.
Moreover, as $S$ is commutative, for any $s\in S$,
\[0\hrel s \;\equ\; 0\lrel s \;\equ\; (0\in\tild\,(S+s)) \eller (s\in\tild\,(S+0)) \;\equ\; 0\in\tild\,(S+s) \;\equ\; \forall_{t\in S} (s+t \apt 0) \]
and hence
\[\tild\,(0\hrel) =\neg(0\hrel) = \{s\in S \mid \neg\forall_{t\in S}(s+t\apt 0)\}\,.\]

Let $x=(1,0,0)\in S$. If $P$ holds then $x+(0,1,0) \approx 0$, and if $\neg P$ holds then $x+(0,0,1)\approx0$; hence $\forall_{t\in S}(x+t\apt 0)$ is impossible, so $x\in\tild\,(0\hrel)$.
On the other hand, invertibility of $[x] \in \tild\,(0\hrel)/{\approx}$ is equivalent to the statement $\exists_{y\in S} \neg(x+y\apt 0)$, that is, to
\begin{equation} \label{nogreeneq}
  \exists_{y_1,y_2,y_3\in\N}\,\neg\left[\,
(P \och ((1+y_1\ne y_2) \eller y_3\ne 0)\,)
\eller (\neg P \och ((1+y_1\ne y_3) \eller y_2\ne 0)\,)
\,\right]
\end{equation}
For any given $y=(y_1,y_2,y_3)\in S$, at least one of the statements
$((1+y_1\ne y_2) \eller y_3\ne 0))$ and $(\neg P \och ((1+y_1\ne y_3) \eller y_2\ne 0))$ must hold. As the inequality relation on $\N$ is decidable, it follows that the statement
$(P \och ((1+y_1\ne y_2) \eller y_3\ne 0)\,)
\eller (\neg P \och ((1+y_1\ne y_3) \eller y_2\ne 0)\,)$
is equivalent to either
\begin{align*}
  &P \quad && \mbox{(if $((1+y_1\ne y_3) \eller y_2\ne 0)$ is false)};\\
  &\neg P && \mbox{(if $((1+y_1\ne y_2) \eller y_3\ne 0)$ is false)};\;\mbox{or}\\
  &P\eller\neg P &&\mbox{(if both are true)}.
\end{align*}
In sum, \eqref{nogreeneq} is equivalent to
$\neg P \eller \neg\neg P \eller \neg(P\eller\neg P)$ and thus to $\neg P \eller \neg\neg P$.

Consequently, the statement that $\tild\,(0\hrel)/{\approx}$ is a group entails WLEM.
\end{ex}

We say that a co-congruence $\kappa$ on $S$ is \emph{idempotent separating} if $e\apt f$ implies $e\kappa f$ for all $e,f\in E(S)$.

\begin{prop}\label{idpsep}
A co-congruence $\kappa$ on a regular semigroup $S$ is idempotent separating if and only if $\hrel\subset \kappa$. 
\end{prop}

\begin{proof}
Assume that $\kappa$ is idempotent separating. Let $(a,b)\in{\hrel}$. This means that least one of the relations $a\succ_r b$, $b\succ_r a$, $a\succ_l b$ and $b\succ_l a$ holds.

Suppose that $a\succ_r b$, and let $a'\in V(a)$. Then $a\succ_r aa' \eller aa'\succ_r ba' \eller ba'\succ_r b$, by cotransitivity of $\succ_r$. But both $a\succ_r aa'$ and $ba'\succ_r b$ are impossible, so $aa'\succ_r ba'$. Set $c=ba'$, and let $x\in V(c^2)$. Then $aa'\succ_r cxc$ or $cxc\succ_r c$, the latter of which is impossible, and hence $aa'\succ_r cxc$. In particular, $aa'\apt cxc$, since $\succ_r$ is strongly irreflexible. As $\kappa$ is idempotent separating, and $aa',cxc\in E(S)$, it follows that $(aa')\kappa(cxc)$, and hence $(aa')\kappa c \eller c\kappa(cxc)$ by the cotransitivity of $\kappa$.

If $c\kappa(cxc)$ then Proposition~\ref{idpprop}\eqref{idpprop3}, applied to the to the factor semigroup $S/(\neg\kappa)$ with apartness $\kappa$, implies that $c\in \tild E(S/(\neg\kappa))$. 
In particular, since $aa'\in E(S)\subset E(S/(\neg\kappa))$, it follows that $(aa')\kappa c$.
Thus, in either case, we have $(aa')\kappa c = ba'$. By co-compatibility, this implies that $a\kappa b$.

By symmetry, each one of the remaining relations $b\succ_r a$, $a\succ_l b$ and $b\succ_l a$ implies $a\kappa b$, too. Thus $\hrel\subset\kappa$, concluding the proof of the ``only if'' part of the proposition.

For the ``if'' part, it suffices to prove that $\hrel$ is idempotent separating, i.e., that $e\apt f$ implies $e\hrel f$ for all $e,f\in E(S)$. 
Let $e,f\in E(S)$, and assume that $e\apt f$. Then, for all $x,y\in S^1$,
\[
\begin{array}{cccccccc}
  & e\apt f^2x &\eller& f^2x\apt fe &\eller& fe\apt ye^2 &\eller& ye^2\apt f \\
  \impl & e\apt fx &\eller& fx\apt e &\eller& f\apt ye &\eller& ye\apt f \\
  \equ &&& e\apt fx &\eller& f\apt ye
\end{array}
\]
that is, $\forall_{x,y\in S^1}(e\apt fx \eller f\apt ye)$. By the constant domains property \eqref{cdinS}, this implies
\[
(\forall_{x\in S^1}:e\apt fx) \eller (\forall_{y\in S^1}:f\apt ye) \,,
\]
that is, $e\succ_r f$ or $f\succ_l e$. In either case, $e\hrel f$. This proves that $\hrel$ is idempotent separating, and thus so is $\kappa$ whenever ${\hrel}\subset{\kappa}$. 
\end{proof}

\begin{rmk}
  It is natural to ask to which extent the constant domains assumption \eqref{cdinS} is fundamentally necessary, and how much of the theory in this section can be made to work without it. As the observant reader will have noticed, \eqref{cdinS} is used, directly or indirectly, in the proofs of almost all the results. Indeed, it seems to us that it would be difficult to develop any meaningful theory for the relations $\lrel$, $\rrel$, $\jrel$, $\drel$ and $\hrel$ in the general setting, at least along the lines of the approach taken in this paper.

The first problem encountered in the general situation would be to give suitable definitions. Take for example the the relation
${\lrel} = (\succ_l) \cup (\succ_l)^{-1} = \tild\,(\le_L)\cup \tild\,(\le_L)^{-1}$.
The equality of the two presentations of $\lrel$ depends on \eqref{cdinS}, and in the general case either one of them could be a candidate for the definition of $\lrel$, together with other possibilities such as, for example, $\tild\Lrel$.
In any case, it seems to us that, in order to be a meaningful analogue of Green's left equivalence $\Lrel = (\le_L)\cap (\le_L)^{-1}$, the relation $\lrel$ should at least satisfy the inclusions
\[
\tild\,(\le_L)\cup \tild\,(\le_L) ^{-1} \subset \lrel \subset \neg{\Lrel} .
\]

However, it turns out that under these assumptions, it is impossible to prove in general that $\lrel$ is a co-quasiorder.
As in Remark~\ref{notsermk}, consider again the multiplicative monoid $(\R,\cdot)$. Since $b\Lrel 1$ holds whenever $b\apt0$, the inclusion $\hrel\subset\neg\Lrel$ implies that $\lrel1\subset \neg(\apt0) = \{0\}$. Conversely, let $a,b\in \R$ be such that $a\le_L b$, that is, $a = sb$ for some $s\in\R$. Since $1\apt0$, cotransitivity implies that either $1\apt a$ or $sb\apt 0= s0$ holds. In the latter case, $b\apt 0$, and thus we have either $1\apt a$ or $0\apt b$. This proves that $(1,0)\apt(a,b)$ whenever $a\le_Lb$, that is, $(1,0)\in \tild\,(\le_L)$. 
Now the assumption $\tild\,(\le_L)\cup \tild\,(\le_L) ^{-1} \subset \lrel$ implies that $1\lrel 0$, and it follows that ${\lrel1} = \{0\}$.
On the other hand, if $\lrel$ is a co-quasiorder then ${\lrel1}\subset \R$ is strongly extensional, by Propositions~\ref{kappaisse} and \ref{serelations}. But strong extensionality of $\{0\}$ in $\R$ is equivalent to LPO, and thus cannot be proved in a general constructive framework. 
\end{rmk}

\section*{Acknowledgements}

ED was partially supported by JSPS Grant-in-Aid for Scientific Research (C) 18K03238, and MM by the Ministry of Education, Science and Technological Development of the Republic of Serbia, contract no.~451-03-9/2021-14/200109.
This work was initiated during a visit by ED to the Faculty of Mechanical Engineering at the University of Ni\v{s} in the winter 2015-16, supported by a scholarship from the Erasmus Mundus EUROWEB+ project. ED wishes to express his gratitude to the colleagues at the Faculty of Mechanical Engineering and the Centre of Applied Mathematics for their hospitaly during that stay.

\bibliographystyle{abbrv}
\bibliography{../litt}

\end{document}